\documentclass[11pt]{article}

\usepackage{amssymb, amsmath}

\usepackage[english]{babel}

\usepackage{amsmath}
\usepackage{amsfonts}
\usepackage{amsthm}
\usepackage{bm,amssymb}
\newcommand{\T}{\mathcal{T}}

\newcommand{\A}{\mathcal{A}}

\newcommand{\C}{\mathcal{C}}
\newcommand{\I}{\mathfrak{J}}

\newcommand{\Vt}{\mathcal{V}}
\newcommand{\U}{\mathcal{U}}

\newcommand{\f}{\mathbf{F}}

\newcommand{\ep}{\epsilon}

\def\R{\ensuremath{\mathrm{I\!R}}}
\def\N{\ensuremath{\mathrm{I\!N}}}

\def\bmu{{\bar \mu}}
\def\bnu{{\bar \nu}}
\def\bt{{\bar t}}
\def\bs{{\bar s}}
\def\d{{\bf d}}

\newcommand{\be}{\begin{equation}}
\newcommand{\ee}{\end{equation}}
\newcommand{\ben}{\begin{equation*}}
\newcommand{\een}{\end{equation*}}
\def\lg{\langle}
\def\rg{\rangle}

\def\W{{\mathcal W}}
\def\p{{\bf p}}
\def\q{{\bf q}}
\def\V{{\bf V}}
\def\u{{\bf U}}

\def\ds{\displaystyle}

\newtheorem{Theorem}{Theorem}[section]
\newtheorem{Definition}[Theorem]{Definition}
\newtheorem{Proposition}[Theorem]{Proposition}
\newtheorem{Lemma}[Theorem]{Lemma}

\newtheorem{Remark}[Theorem]{Remark}

\DeclareMathOperator{\esssup}{ess-sup}

\title{A differential game with a blind player}
\author{Pierre Cardaliaguet\thanks{
CEREMADE (UMR CNRS 7534), Universit\'e Paris-Dauphine. E-mail:
cardaliaguet@ceremade.dauphine.fr} and Anne Souqui\`{e}re\thanks{Institut TELECOM ; TELECOM Bretagne ; UMR CNRS 3192 Lab-STICC,
Technople Brest Iroise, CS 83818, 29238 Brest Cedex 3. E-mail: anne.souquiere@telecom-bretagne.eu}
\thanks{Laboratoire de Math\'ematiques (UMR CNRS 6205), universit\'e de Brest.}}

\date{\today}

\begin{document}
%==============================================================================
\maketitle
\begin{abstract}
We consider a zero sum differential game with lack of observation on one side. The initial state of the system is drawn at random according
to some probability $\mu_0$ on $\R^N$. Player I is informed of the initial position of state while player II knows only $\mu_0$. 
Moreover Player~I observes Player~II's moves while Player~II is blind and has no further information. 
We prove that in this game with a terminal payoff the value  exists and is characterized as 
the unique viscosity solution of some Hamilton-Jacobi equation on a space of probability measures.
\vspace{3mm}
\begin{center}
\end{center}

\end{abstract}

%\noindent{\bf A.M.S. classification :} 49N70, 91A23.
%\vspace{3mm}

%\begin{acknowledgement}
%Thanks everybody
%\end{acknowledgement}

%===============================================================================
\section*{Introduction}
%==============================================================================
We consider a two player zero sum differential game in $\R^N$ with finite horizon $T>0$.
Its dynamics is given by :
\begin{equation}
\label{dynamique}
\left\{ \begin{array}{ll}
x'(t)=f(x(t),u(t),v(t))\;, & t\in [t_0,T],\ u(t) \in U ,\ v(t) \in V\\
x(t_0)=x_0 &\\
\end{array}
\right.
\end{equation}
where Player I uses the measurable control $u\in\U(t_0):=L^1([t_0,T],U)$ and Player II the measurable control $v\in\Vt(t_0):=L^1([t_0,T],V)$.
We denote by $(X_\cdot^{t_0,x_0,u,v})$ the solution of \eqref{dynamique}, which is unique under suitable assumptions on $f$ stated below.
In this zero sum game, Player I aims at minimizing a final cost $g(x(T))$, where $g:\R^N\to \R$, while player II aims at maximizing it.
We introduce lack of observation in the following way:
\begin{itemize}
\item At time $t_0$, the initial state of the system, $x_0$, is drawn at random according to some probability measure $\mu_0$ on $\R^N$; 
\item Player I is informed of $x_0$ while Player II is only informed of $\mu_0$;
\item During the game, Player II observes neither the state of the system and nor the control played by his or her opponent, 
while Player I has a full information on the control played so far by Player~II (and therefore on the state of the system as well).
\end{itemize}

%Because of the lack of observation, Players use random strategies. Player I will play some random non anticipative strategy with delay depending on the initial state of the system $\alpha(x)$ and Player II will use some random open-loop control $v$. We introduce the upper and lower value functions:
%\[\V^+(t,\mu)=\inf_{\alpha}\sup_{v}\int_{\R^N}\E_{\alpha(x) v}\left[g(X_{T}^{t,x,\alpha(x),v})\right]d\mu(x)\]
%and, denoting by $\{\alpha_\tau\}$ the non anticipative strategies with fixed delay $\tau$:
%\[\V^-(t,\mu)=\lim_{\tau\rightarrow 0^+}\sup_{v}\inf_{\alpha_\tau}\int_{\R^N}\E_{\alpha_\tau (x) v}\left[g(X_{T}^{t,x,\alpha_\tau(x),v})\right]d\mu(x)\;.\]

Our aim is to prove that the game has a value and to characterize this value as the unique viscosity solution of some Hamilton-Jacobi equation. Since the natural
state of the system is the space of probability measures on $\R^N$, this Hamilton-Jacobi equation takes place in this space. 

%The problem studied in this paper is related with some previous works: 
Games with asymmetric information were studied---mostly on examples---by several authors: see for instance
Bernhard and Rapaport \cite{Bernhard}, Gal \cite{Gal79}, Petrosjan \cite{Petrosjan} and Baras and James \cite{Baras 96}. 
In this latter reference the authors introduce an underlying Hamilton-Jacobi equation in
some infinite dimensional space, but, since the game is seen as a control problem with disturbance, 
the value function considered there differs considerably from ours. 
%As for games with incomplete information---i.e., games in which 
%one of the players has private information on the system but both players have full knowledge of their oponent's moves, 
%they are studied in \cite{Cardaliaguet asy}, paper strongly inspired with the Aumann-Maschler repeated game \cite{}. 

Our game has actually much to do with a previous work of Cardaliaguet and Quincampoix \cite{quincampoix} which analyses problems in which the only information that both  players have on the initial position of the system is that it has been randomly choosen according to some probability known to both players. In 
\cite{quincampoix}, the players observe each other. This is a main difference with our problem, where the lack of observation of one player induces the use of completely asymmetric strategies. The introduction of a suitable notion of strategies to formalize this situation is one of the novelties of our paper. A dramatic consequence of the asymmetry of information is that the usual machinery of differential games (dynamic programming, which leads to the characterization of the value functions as the unique solution of some Hamilton-Jacobi equation) does not work. Indeed the lower value does not seem to satisfy any dynamic programming, because 
the uninformed player cannot actualize his  or her strategy along the game since he or she sees nothing. However, and fortunately, it turns out that, in the game for the upper value, the uninformed player, knowing the strategy of his or her oponent, can actualize his or her own strategy along the time. This leads to a dynamic programming for the upper value, which takes place in the space of probability measures on $\R^N$. From this we derive that the upper value satisfies a Hamilton-Jacobi equation in some suitable viscosity sense. The definition of the viscosity solution in this framework is inspired by, but slightly differs from, the one given in \cite{quincampoix}. Other definitions of viscosity solution in the Wasserstein space have been used in the literature, in general for more singular dynamics (see for instance \cite{FK09, FS10, GNT, LcoursColl}). Then the existence of a value (i.e., the fact that the upper value coincides with the lower one) relies on min-max arguments combined with PDE ones: we introduce an auxiliary game in which the uninformed player chooses a strategy by randomizing over a finite set of controls. Existence of a value for this game is obtained by min-max arguments. 
Then we show that this auxiliary game is close to the continuous one by using techniques from Crandall and Lions \cite{Crandall Lions infty} on the stability of
 viscosity solutions for Hamilton-Jacobi equations in infinite dimension.

The paper is organized in the following way: in the first section, we define the strategies and state the assumptions on the game. In the second section, we prove that the upper and lower value functions are Lipschitz continuous. In the third section, we state the dynamic programming principle for the upper value function. In section $4$, we show that if a function satisfies this dynamic programming principle, then it is the unique viscosity solution of some Hamilton-Jacobi equation. In section $5$, we introduce some discrete approximation for the game. In the last section, we prove, using the discrete game, that the game has a value. 

%%%%%%%%%%%%%%%%%%%%%%%%%%%%%%%%%%%%%%%%%%%%%%%%%%%%%%%%%%%%%%%%%%%%%%%%%%%%%%%%%%%%%%%%%%%%%%%%%%
%%%%%%%%%%%%%%%%%%%%%%%%%%%%%%%%%%%%%%%%%%%%%%%%%%%%%%%%%%%%%%%%%%%%%%%%%%%%%%%%%%%%%%%%%%%%%%%%%%
%%%%%%%%%%%%%%%%%%%%%%%%%%%%%%%%%%%%%%%%%%%%%%%%%%%%%%%%%%%%%%%%%%%%%%%%%%%%%%%%%%%%%%%%%%%%%%%%%%

\section{Definitions and assumptions}
\label{Definitions and assumptions}
We first introduce some notations on the space of probability measures on $\R^N$.
For a fixed closed subset $K$ of $\R^N$ we denote by $\W(K)$ the set of Borel probability measures with support included in $K$ and with finite second order moment.
We set $\W=\W(\R^N)$. 
For any $\mu, \nu \in \W$, let $\Pi(\mu,\nu)$ be the set of probability measures on $\R^{2N}$ with first marginal $\mu$ and second marginal
$\nu$. Recall that the Wasserstein distance on $\W$ between $\mu$ and $\nu$ is defined as
$$
{\bf d}^2(\mu,\nu)=\inf_{\pi\in \Pi(\mu,\nu)}\int_{\R^{2N}} |x-y|^2 d\pi(x,y)\;.
$$
It is well-known that this infimum is in fact a minimum and we denote by $\Pi_{opt}(\mu,\nu)$ the set of minimizers in the above minimization problem.
If $\mu$ is a probability measure on a set $X$ and $\varphi:X\to Y$, we denote by $\varphi\sharp \mu$ the 
pushforward image of $\mu$ by $\varphi$, defined by $\varphi\sharp \mu( A)= \mu( \varphi^{-1}(A))$ for any subset $A$ of $Y$ for which this definition makes sense. 

Next we introduce notations and assumptions related to the game.
The payoff only depends on the terminal state of the system. More precisely, if at time $T$ the system is at some position $x(T)$, then 
the outcome of the game is $g(x(T))$, where $g:\R^N\to \R$ is a fixed Lipschitz continuous and bounded function. 
Assume that the initial state is choosen at random according to 	a probability measure 
 $\mu_0\in \W$ at the initial time $t_0$ and suppose for a while that the players use a pair of controls $(u,v)\in \U(t_0)\times \Vt(t_0)$ 
 independent of the initial state. Then the outcome of the game is 
 \[\I(t_0,\mu_0,u,v)=\int_{\R^N}g(X_T^{t_0,x,u,v})d\mu_0(x)\;.\]
Throughout this paper we tacitely assume that the following conditions on the data are satisfied:
\begin{equation}
\label{régularité U,V,f}
\left \{
\begin{array}{lll}
i) & U \text{ and } V \text{ are compact subsets of some finite dimensional vector spaces,}\\
ii) & f \text{ is bounded, uniformly continuous on $\R^N\times U\times V$,}\\
& \text{ and uniformly Lipschitz continuous with respect to the $x$ variable,}\\
iii)& g \text{ is Lipschitz continuous and bounded.}
\end{array}
\right.
\end{equation}
During the proofs, we denote by $C$ a generic constant depending on $N$, $f$ and $g$. 

%Moreover we will also need the following assumptions on the dynamics
%\be\label{faffine}
%V\; \mbox{\rm is convex and } f(x,u,v)=f_u(x,u)+f_v(x)v\;,
%\ee
%where $f_u$ and $f_v$ are Lipschitz continuous with respect to $x$.\\

For any $0\leq t_0<t_1\leq T$ we denote by $\U(t_0,t_1)$ the set of Lebesgue measurable maps $u:[t_0,t_1]\to U$. 
We abbreviate the notation into $\U(t_0)$ whenever $t_1=T$. 
We endow $\U(t_0,t_1)$ with the $L^1$ distance 
$$
d_{\U(t_0,t_1)}(u_1,u_2)= \int_{t_0}^{t_1} |u_1(s)-u_2(s)|ds \qquad \forall u_1,u_2\in \U(t_0,t_1)\;,
$$
and with the Borel $\sigma-$algebra associated with this distance. Recall that $\U(t_0,t_1)$ is then a Polish space (i.e., a complete separable 
metric space). We denote by $\Delta( \U(t_0,t_1))$ the set of Borel probability measures on $\U(t_0,t_1)$. This set is
endowed with the weak-* topology, for which there is an associated distance defined  as follows:
%It is endowed with the 
%distance 
$$
d_{\Delta( \U(t_0,t_1))}(P_1,P_2)= \sup\left\{ \int_{\U(t_0,t_1)} \varphi(u)d(P_1-P_2)(u)\;, \; \right\}\;,
$$
where the supremum is taken over the set of Lipschitz continuous maps $\varphi: \U(t_0,t_1)\to [-1,1]$ with a Lipschitz constant less than $1$.

The sets $\Vt(t_0,t_1)$ and $\Vt(t_0)$ of Lebesgue
measurable maps $v:[t_0,t_1]\to V$ and $v:[t_0,T]\to V$ are defined in a symmetric way and 
endowed with the $L^1$ distance and with the associate Borel $\sigma-$algebra. The set of Borel probability measures on $\Vt(t_0,t_1)$ is denoted 
by $\Delta( \Vt(t_0,t_1))$.

We say that a map $(x,v) \to P^v_x$ from $\R^N\times \Vt(t_0)$ into $\Delta( U(t_0))$ is measurable if, for any Borel subset $A$ of $\U(t_0)$, the mapping
$(x,v)\to  P^v_x(A)$ is Borel measurable. 
 
\begin{Definition} A strategy for Player I for the initial time $t_0\in [0,T]$ is a  measurable mapping 
$(x,v)\to P^v_x$ from $\R^N\times\Vt(t_0)$ into $\Delta( \U(t_0))$ which fulfills the following nonanticipativity condition: 
there is some delay $\tau>0$ such that, if two controls $v_1,v_2\in \Vt(t_0)$ coincide a.e. on $[t_0,t]$ for some $t\in [t_0,T]$
 and if $R_{t_0,(t+\tau)\wedge T}$ denotes the restriction mapping from $\U(t_0)$ onto $\U(t_0, (t+\tau)\wedge T)$, then the measures
$R_{t_0,(t+\tau)\wedge T} \sharp P^{v_1}_x$ and $R_{t_0,(t+\tau)\wedge T} \sharp P^{v_2}_x$
coincide (on $\U(t_0,(t+\tau)\wedge T)$) for any $x\in \R^N$. 
\end{Definition}
Note that  strategies for Player~I actually correspond to behavioral strategies in game theory, because Player I adapts 
his or her probability measure in function of the past
behaviour of his or her oponent. 
The heuristic interpretation
of a strategy $P$ is that, if the state of the system is at the initial position $x$, then Player~I answers (in a nonanticipative way)
to a control $v\in \Vt(t_0)$ played by Player~II a control $u \in\U(t_0)$ with probability $P^v_x(u)$.

We denote by $\Delta( A_x(t_0))$ the set of strategies for Player~I, by $\Delta( \A_x^\tau(t_0))$ the set of such strategies
 which have a delay  $\tau$ and by 
$\A_x^{\tau}(t_0)$ the subset of $\Delta( \A_x^\tau(t_0))$ consisting in {\it deterministic} strategies, i.e., 
strategies for which, for any $(x,v)\in \R^N\times \Vt(t_0)$, $P^v_x$ is a Dirac mass. If $P\in \A_x^\tau(t_0)$, then there 
is a map $\alpha: \R^N\times \Vt(t_0)\to \U(t_0)$ such that $dP^v_x(u)=d\delta_{\alpha(x,v)}(u)$, and this map satisfies the nonanticipative
property: for $\mu_0-$a.e. $x\in \R^N$, if two controls $v_1,v_2\in \Vt(t_0)$ coincide a.e. in $[t_0,t]$ for some $t\in [t_0,T]$, then 
$\alpha(x,v_1)=\alpha(x,v_2)$ a.e. in $[t_0,(t+\tau)\wedge T]$. Generic elements of $\A_x^{\tau}(t_0)$ are systematically identified with
the maps $\alpha$.

Since Player~II observes neither the state nor his or her oponent behavior, the definition of  his or her strategies is much simpler than 
for Player~I:

\begin{Definition}
A strategy for Player II is a Borel probability measure $Q$ on the set $\Vt(t_0)$. 
\end{Definition}

Recall that we denote by $\Delta( \Vt(t_0))$ the set of such  strategies. 
Given $Q\in \Delta( \Vt(t_0))$ and $P\in \Delta( \A_x(t_0))$ we denote by $\I(t_0,\mu_0, P, Q)$ the outcome of the two
strategies $P$ and $Q$: 
$$
\I(t_0,\mu_0, P, Q)= \int_{\R^N\times \U(t_0)\times \Vt(t_0)} g\left(X_T^{t_0,x,u,v}\right)
dP^v_x(u)dQ(v)d\mu_0(x)\;.
$$

\noindent We are now ready to define the value functions. 
The {\it lower value} of the game is:
$$
\begin{array}{rl}
\V^-(t_0,\mu_0)\; =& {\ds \lim_{\tau\rightarrow 0^+}\sup_{Q \in\Delta(\Vt(t_0))}\inf_{P\in \Delta( \A_x^\tau(t_0))}
\I(t_0,\mu_0,P,Q) }\\
& =\ds{ \lim_{\tau\rightarrow 0^+}\V^-_\tau(t_0,\mu_0) }
\end{array}
$$
where we have set 
\be\label{defV-tau}
\begin{array}{rl}
\V^-_\tau(t_0,\mu_0)\; = & \ds{ \sup_{Q \in\Delta(\Vt(t_0))}\inf_{P\in \Delta( \A_x^\tau(t_0))}
\I(t_0,\mu_0,P,Q)}\\
= & {\ds \sup_{Q \in\Delta(\Vt(t_0))}\inf_{\alpha\in \A_x^\tau(t_0)}
\I(t_0,\mu_0,\alpha,Q)}
\end{array}
\ee

\noindent The {\it upper value} of the game is defined in a symmetrical way:
\begin{equation*}
\begin{split}
\V^+(t_0,\mu_0)&=\lim_{\tau\to 0} 
\inf_{P\in\Delta( \A_x^\tau(t_0))}\sup_{Q \in\Delta(\Vt(t_0))}
\I(t_0,\mu_0,P,Q)\\
&=\inf_{P\in\Delta( \A_x(t_0,\mu_0))}\sup_{v \in\Vt(t_0)}\I(t_0,\mu_0,P,v)
\end{split}
\end{equation*}

%\begin{Remark}{\rm 
%Note that the lower value function is not defined as \[V^-(t_0,\mu_0)=\sup_{Q\in\Delta(\Vt(t_0))}\inf_{P\in \A_x(t_0,\mu_0)}
%\I(t_0,\mu_0,P,Q)\;.\]
%Indeed, in this case, Player I could ``synchronize" with his or her opponent, which would destroy the random character of this latter's control. 
%EXPLIQUER MIEUX
%}\end{Remark}

%%%%%%%%%%%%%%%%%%%%%%%%%%%%%%%%%%%%%%%%%%%%%%%%%%
\section{Regularity of the value functions}
%%%%%%%%%%%%%%%%%%%%%%%%%%%%%%%%%%%%%%%%%%%%%%%%%%
We begin by proving the Lipschitz continuity of the upper and lower value functions, which is important for the characterization of the value as a viscosity solution of some Hamilton-Jacobi equation.

\begin{Proposition}[Regularity of the value functions]
\label{V+V- Lipschitz}
The value functions $\V^+$ and $\V^-$ are Lipschitz continuous on $[0,T]\times \W$.
\end{Proposition}
\begin{proof}
We start with the Lipschitz continuity of $\V^+$ with respect to the $\mu$ variable. Let $t_0\in [0,T]$, $\mu,\nu\in \W$ and 
choose $\gamma\in\Pi_{opt}(\nu,\mu)$ some optimal transport plan between $\mu$ and $\nu$. Let us recall that $\gamma$ admits a 
desintegration of the form $d\gamma(x,y)= d\gamma_x(y)d\nu(x)$ where the map $x\to \gamma_x$ is measurable, i.e., such that the map
$x\to \gamma_x(A)$ is Borel measurable for any Borel set $A\subset \R^N$. 
Let $P\in\Delta( \A_x(t_0))$ be an $\epsilon$-optimal strategy for $\V^+(t_0,\mu)$, i.e., $P$ satisfies
$$
\sup_{v\in\Vt(t_0)} \I(t_0,\mu,P,v) \leq \V^+(t_0,\mu)+\ep\;.
$$
We define the strategy $\tilde P\in\Delta( \A_x(t_0))$ by  
 $$
\int_{\U(t_0)}\varphi(u) d\tilde P^v_x(u)=\int_{\R^{N}\times \U(t_0)} \varphi(u) dP^v_y(u)d\gamma_x(y)
 $$
for any $(x,v)\in \R^N\times \Vt(t_0)$ and  for any nonnegative Borel measurable map $\varphi:\U(t_0)\to\R$.
Let now $v\in \Vt(t_0)$ and let us estimate $\I(t_0,\nu, \tilde P,v)$: we have
$$
\begin{array}{rl}
\I(t_0,\nu, \tilde P,v) \; = &  \ds{ \int_{\R^{2N}\times \U(t_0)} g\left(X_T^{t_0,x,u,v}\right)dP^{v}_y(u)d\gamma_x(y)d\nu(x) }\\
\leq  & \ds{ \int_{\R^{2N}\times \U(t_0)} \left[  g\left(X_T^{t_0,y,u,v}\right)+C|x-y|\right] dP^{v}_y(u)d\gamma(x,y) }\\ 
\leq & \ds{   \int_{\R^N\times \U(t_0)}  g\left(X_T^{t_0,y,u,v}\right)dP^{v}_y(u)d\mu(y)+C
\int_{\R^N\times \R^N}|x-y|d\gamma(x,y) }\\ 
\leq & \V^+(t_0,\mu)+\ep+ C\d(\mu,\nu)
\end{array}
$$
Therefore 
$$
\V^+(t_0,\nu)\leq   \sup_{v\in \Vt(t_0)} \I(t_0,\nu, \tilde P,v) \leq  \V^+(t_0,\mu)+\ep+ C\d(\mu,\nu)\;.
$$
This proves the Lipschitz continuity of $\V^+$ with respect to second variable, uniformly with respect to the time variable. \\

 We now prove that $\V^+$ is Lipschitz continuous with respect to the time variable. 
Fix $t_0<t_1\leq T$,  $\mu\in\W$ and $v_0\in\Vt(t_0)$. We choose some $\epsilon$-optimal strategy $P$ for Player~I in $\V^+(t_0,\mu)$ and 
define the  strategy $\tilde P \in\Delta( \A_x(t_1))$ by:
$$
\int_{\U(t_1)} \varphi(u_1)d\tilde P^v_x(u_1)=\int_{\U(t_0)} \varphi(u_{|_{[t_1,T]}})dP^{(v_0,v)}_x(u)
$$
where $(v_0,v)$ denotes the concatenation of the controls $v_0$ and $v$,  for any $(x,v)\in \R^N\times \Vt(t_1)$ 
and any nonnegative Borel measurable map
$\varphi:\U(t_1)\to \R$. 
Then, for any  $v\in \Vt(t_1)$, we have
$$
\begin{array}{rl}
\I(t_1,\mu, \tilde P, v) \; 
%= & \ds{ \int_{\R^N\times \U(t_1)}  g\left( X_T^{t_1,x, u_1,v}\right) d\tilde P^v(x,u_1) } \\
= & \ds{  \int_{\R^N\times \U(t_0)}  g\left( X_T^{t_1,x, u_{|_{[t_1,T]}}, v}\right) d P^{(v_0,v)}_x(u)d\mu(x)  } \\
\leq & \ds{  \int_{\R^N\times \U(t_0)}  \left[ g\left( X_T^{t_0,x, u, (v_0,v)}\right)  +C(t_1-t_0) \right]d P^{(v_0,v)}_x(u)d\mu(x)   }  \\
\leq & \V^+(t_0, \mu)+\ep+C(t_1-t_0)
\end{array}
$$
Therefore we get:
\[\V^+(t_1,\mu)-\V^+(t_0,\mu)\leq \epsilon+C(t_1-t_0)\;.\]

For the reverse inequality, let $P\in \Delta( \A_x(t_1))$ be some $\epsilon$-optimal strategy for player~I in $\V^+(t_1,\mu)$. 
We fix $u_0\in \U(t_0)$ and define the strategy $\tilde P\in \Delta( \A_x(t_0))$ by
$$
\int_{\U(t_0)} \varphi(u)d\tilde P^v_x(u)=\int_{\U(t_1)} \varphi((u_0,u_1))dP^{v_{|_{[t_1,T]}}}_x(u_1)
$$
for any $(x,v)\in \R^N\times \Vt(t_0)$ and for any nonnegative Borel measurable map
$\varphi:\U(t_0)\to \R$.
Then, for any $v\in \Vt(t_0)$, we have 
$$
\begin{array}{rl}
\I(t_0,\mu, \tilde P, v) \; 
%= & \ds{ \int_{\R^N\times \U(t_0)}  g\left( X_T^{t_1,x, u,v}\right) d\tilde P^v(x,u) } \\
= & \ds{  \int_{\R^N\times \U(t_1)}  g\left( X_T^{t_0,x, (u_0,u_1), v}\right) d P^{v_{|_{[t_1,T]}}}_x(u_1) d\mu(x) } \\
\leq & \ds{  \int_{\R^N\times \U(t_1)}  \left[ g\left( X_T^{t_1,x, u_1, v_{|_{[t_1,T]}}}\right) +C(t_1-t_0)\right]d P^{v_{|_{[t_1,T]}}}_x(u_1)  d\mu(x)   }  \\
\leq & \V^+(t_1, \mu)+\ep+C(t_1-t_0)
\end{array}
$$
Hence:
\[\V^+(t_0,\mu)-\V^+(t_1,\mu)\leq \epsilon+C(t_1-t_0)\;.\]
which shows  that $\V^+$ is Lipschitz continuous with respect to the time variable, uniformly with respect to the $\mu$ variable, since $\epsilon$ is arbitrary.\\ 

The proof of the Lipschitz continuity for $\V^-$ goes along the same lines, so we omit it.
\end{proof}

%%%%%%%%%%%%%%%%%%%%%%%%%%%%%%%%%%%%%%%%%%%%%%%%%%%%%%%%%%%%%%%%%%%%%%%%%%%%%
%%%%%%%%%%%%%%%%%%%%%%%%%%%%%%%%%%%%%%%%%%%%%%%%%%%%%%%%%%%%%%%%%%%%%%%%%%%%%
%%%%%%%%%%%%%%%%%%%%%%%%%%%%%%%%%%%%%%%%%%%%%%%%%%%%%%%%%%%%%%%%%%%%%%%%%%%%%
%%%%%%%%%%%%%%%%%%%%%%%%%%%%%%%%%%%%%%%%%%%%%%%%%%%%%%%%%%%%%%%%%%%%%%%%%%%%%
\section{Dynamic programming for the upper value function}

We prove in this section that $\V^+$ satisfies some dynamic programming principle.
We have to define how Player II's information evolves in time. In the game $\V^+(t_0,\mu_0)$, 
Player~II knows the initial distribution of the state variable as well as  his or her opponent's strategy $P$. If he or she plays the control $v\in\Vt(t_0)$,
his or her information on the state of the system at time $t_1\in(t_0,T]$ is the probability measure $\mu_{t_1}^{t_0,\mu_0,P,v}$ defined by:
\[\forall \varphi\in \C_b(\R^N,\R),\ \int_{\R^N} \varphi(x)d\mu_{t_1}^{t_0,\mu_0,P,v}(x)=\int_{\R^N\times\U(t_0)} \varphi(X_{t_1}^{t_0,x,u,v})d
P^v_x(u)d\mu_0(x)\;.\]
Note that $\mu_{t_1}^{t_0,\mu_0,P,v}$ belongs to $\W$.

\begin{Proposition}[Dynamic programming principle for $\V^+$]\label{dyn prog}
For any $(t_0,t_1,\mu_0)$ such that $t_1\in(t_0,T]$, we have:
\[\V^+(t_0,\mu_0)=\inf_{P\in \Delta( \A_x(t_0))}\sup_{v \in\Vt(t_0)}\V^+(t_1,\mu_{t_1}^{t_0,\mu_0,P,v})\;.\]
\end{Proposition}
\begin{proof}
We denote by $W(t_0,t_1,\mu_0)$ the  right-hand side of the previous equality. Arguing as for Proposition \ref{V+V- Lipschitz}, one can show that $W$ is Lipschitz 
continuous with respect to the measure variable. 

Let us now show that we can assume in addition to (\ref{régularité U,V,f}) that $f$ has a uniformly bounded ${\mathcal C}^2$ norm with respect to the $x$ variable. 
Indeed, from our assumptions on $f$,
if we mollify $f$ with respect to the $x$ variable, we obtain a sequence of uniformly continuous  functions $f_n:\R^N\times U\times V\to \R^N$, with a modulus of continuity
independent of $n$, uniformly (with respect to $n$) Lipschitz continuous in space and  which converge uniformly  to $f$ on $\R^N\times U\times V$. We easily check that the upper value function $\V_n^+$ for $f_n$ corresponding to $f_n$ converges to $\V^+$ and that the $W_n$ converge to $W$
uniformly on $[0,T]\times \W$. We also note that the transported measures $\mu_{t_1}^{n,t_0,\mu_0,P,v}$
for $f_n$ converges in $\W$ to $\mu_{t_1}^{t_0,\mu_0,P,v}$ uniformly with respect to $P$ and $v$. So, if Lemma \ref{dyn prog} holds for the $f_n$, it also
holds for $f$. Therefore we can assume, from now on, that $f$ has a uniformly bounded ${\mathcal C}^2$ norm with respect to the $x$ variable. 

Let us first prove that $\V^+(t_0,\mu_0)\leq W(t_0,t_1,\mu_0)$ under the additional assumptions that  $\mu_0\in\W(K)$, 
where $K$ is some compact in $\R^N$. This extra assumption is removed later.
The first step consists in regularizing $\mu_0$. Let $\rho\in {\mathcal C}^\infty_c(\R^N)$ be a smooth mollifier: $\rho\geq0$ is even, has a support in the unit ball
and satisfies $\int_{\R^N}\rho(x)dx=1$. Let  $\rho_\epsilon(x)=\ep^N \rho(\frac{x}{\epsilon})$,
$f_\ep=\rho_\ep*\mu_0$ and $\mu_\ep=f_\ep dx$. 
%Note that $\mu_\ep\in \W(K_\ep)$, where  $K_\ep$ is the $\ep-$neighbourhood of $K$: $K_\ep=\{x\in \R^N,\; d(x,K)\leq \ep\}$.
By standard arguments, we have that $\d(\mu_0,\mu_\ep)\leq \ep$. 

\begin{Lemma}
\label{Lemme approx mu alpha} For any strategy $P\in\Delta(  \A_x(t_0))$, 
there exists a strategy $P_\epsilon\in\Delta( \A_x(t_0))$, with the same delay as $P$, a compact set $K_1\subset \R^N$ and a constant $C$ such that, 
for any $v\in\Vt(t_0)$ and any $t\in [t_0,T]$,

\begin{enumerate}
    \item $\d\left(\mu_{t}^{t_0,\mu_0,P,v},\mu_{t}^{t_0,\mu_\epsilon,P_\epsilon,v}\right)\leq C\epsilon$,
        \item $\mu_{t}^{t_0,\mu_\epsilon,P_\epsilon,v}$ has
        a support in $K_1$ and a density $f^{v,\ep}_t$ bounded in ${\mathcal C}^1(K_1)$ by $C$,
%        \item the map $t\to f^{v,\ep}_t$ is continuous in $L^1(K_1)$, uniformly  with respect to $v$. 
\end{enumerate}
\end{Lemma}

\begin{proof} Let $P_\ep\in \A_x(t_0)$ 
be defined by: if $f_\ep(x)>0$, then we set
$$
\int_{\U(t_0)} \varphi(u)dP^v_{\ep,x}(u)= \frac{1}{f_\ep(x)}\int_{\R^{N}\times \U(t_0)} \varphi(u)\rho_\ep(x-y)dP^v_y(u)d\mu_0(y)
$$
for any $(x,v)\in \R^N\times \Vt(t_0)$ and  for any nonnegative Borel measurable map $\varphi:\U(t_0)\to\R$.
If $f_\ep(x)=0$, we just set $dP^v_{\ep,x}(u)=dP^v_{x}(u)$.

Since $\mu_\epsilon$ has bounded support and  the dynamics is bounded, there is a compact set $K_1$ such that 
 $\mu_{t}^{t_0,\mu_\epsilon,P_\epsilon,v}$ has a support contained in $K_1$ for any $v\in \Vt(t_0)$ and any $t\in[t_0,T]$.

We now compare $\mu_{t}^{t_0,\mu_0,P,v}$ to $\mu_{t}^{t_0,\mu_\epsilon,P_\epsilon,v}$ for any $v\in\Vt(t_0)$ and any $t\in[t_0,T]$: we have
\begin{multline*}\d^2(\mu_{t}^{t_0,\mu_0,P,v},\mu_{t}^{t_0,\mu_\epsilon,P_\epsilon,v})
%\\
\leq \int_{\R^{2N}\times\U(t_0)}\left|X_{t}^{t_0,x,u,v}-X_{t}^{t_0,y,u,v}\right|^2\rho_\epsilon(y-x)dP^v_x(u)d\mu_0(x)dy
\end{multline*}
because the probability measure $\gamma$ on $\R^{2N}$ defined by
\[\int_{\R^{2N}}\varphi(x,y)d\gamma(x,y)=\int_{\R^{2N}\times\U(t_0)}\varphi(X_{t}^{t_0,x,u,v},X_{t}^{t_0,y,u,v)})\rho_\epsilon(y-x)dP^{v}_x(u)d\mu_0(x)dy\]
satisfies $\gamma\in\Pi(\mu_{t}^{t_0,\mu_0,P,v},\mu_{t}^{t_0,\mu_\epsilon,P_\epsilon,v})$. 
Since $\left|X_{t}^{t_0,x,u,v}-X_{t}^{t_0,y,u,v}\right|\leq C|x-y|$, we get
\ben
\begin{split}
&\d^2(\mu_{t}^{t_0,\mu_0,P,v},\mu_{t}^{t_0,\mu_\epsilon,P_\epsilon,v})\\
&\quad\leq C\int_{\R^{2N}\times\U(t_0)}|x-y|^2\rho_\epsilon(y-x)dP^v_x(u)d\mu_0(x)dy\\
&\quad\leq C\epsilon^2\int_{\R^{2N}\times\U(t_0)}\rho_\epsilon(y-x)dP^{v}_x( u )d\mu_0(x)dy \; \leq \; C\ep^2\;.
\end{split}
\een
Therefore for all $v\in\Vt(t_0)$ and any $t\in[t_0,T]$:
\ben
\label{mu epsilon proche de mu}
\d(\mu_{t}^{t_0,\mu_0, P,v},\mu_{t}^{t_0,\mu_\epsilon, P_\epsilon,v})\leq C\epsilon \;.
\een
We now check that the measure $\mu_{t}^{t_0,\mu_\epsilon, P_\epsilon,v}$ is absolutely continuous and
has a density  bounded  in ${\mathcal C}^1(K_1)$ uniformly with respect to $v$ and $t$. 
We first note that for fixed $(t, u ,v)\in[t_0,T]\times \U(t_0)\times \Vt(t_0)$ the map $\T(t, u ,v):x\mapsto X_{t}^{t_0,x, u ,v}$ is 
of class ${\mathcal C}^2$ with a ${\mathcal C}^2$ inverse because the dynamics $f$ is of class ${\mathcal C}^2$ with respect to the $x$
variable. We denote by  $\T(t, u ,v)^{-1}$ this inverse. 
We have, for all $\varphi\in C^0_b(\R^N,\R)$:
\ben
\begin{split}
&\int_{\R^N}\varphi(x)d\mu_{t}^{t_0,\mu_\epsilon, P_\epsilon,v}(x)\\
%&\qquad=\int_{\R^{2N}\times\U(t_0)}\varphi(X_{t}^{t_0,x, u,v})dP_\ep^{v}(x,u)\\
&\qquad=\int_{\R^{2N}\times\U(t_0)}\varphi(X_{t}^{t_0,x, u,v})\rho_\epsilon(x-y)dP^{v}_y( u )d\mu_0(y)dx\\
&\qquad=\int_{\R^{2N}\times\U(t_0)}\varphi(z)\rho_\epsilon(\T(t,u,v )^{-1}(z)-y)|\det J_{\T(t,u,v )^{-1}}(z)|dP^{v}_y( u )d\mu_0(y) \ dz
\end{split}
\een
Therefore $\mu_{t}^{t_0,\mu_\epsilon, P_\epsilon,v}$ is absolutely continuous with a density given by
\[
f^{v,\ep}_t(z)= \int_{\R^N\times \U(t_0)}\rho_\epsilon(\T(t,u,v )^{-1}(z)-y)|\det J_{\T(t,u,v )^{-1}}(z)|dP^{v}_y( u )d\mu_0(y)\;.\] 
Note that $f^{v,\ep}_t$ is bounded 
in ${\mathcal C}^1$, uniformly with respect to $v$ and $t$, thanks to our assumptions on the dynamics $f$. 
%Moreover, since the map $(t,x)\to f^{v,\ep}_t(x)$ is continuous, bounded and with a support in $K_1$,
%the map $t\to f^{v,\ep}_t$ is continuous in $L^1(K_1)$. It is uniformly continuous with respect to $v$ because of the ${\mathcal C}^1$
%bound which provides compactness. 
\end{proof}

We now proceed in the proof of inequality $\V^+(t_0,\mu_0)\leq W(t_0,t_1,\mu_0)$ under the additional assumption that  $\mu_0\in\W(K)$.
Let $P_0 \in \Delta( \A_x(t_0))$ be an $\ep-$optimal strategy for $W(t_0,t_1,\mu_0)$ and 
$P_\ep \in \Delta( \A_x(t_0))$ be the strategy associated to $P_0$
as in Lemma \ref{Lemme approx mu alpha}. 

We first note that $P_\ep$ is $C\ep-$optimal for $W(t_0,t_1,\mu_\ep)$: indeed we have
$$
\begin{array}{rl}
\ds{ \sup_{v\in \Vt(t_0)} \V^+(t_1, \mu_{t_1}^{t_0, \mu_\ep,  P_\ep, v}) }
 \leq & \ds{ \sup_{v\in \Vt(t_0)} \V^+(t_1, \mu_{t_1}^{t_0, \mu_0,  P_0, v}) +C \d(\mu_{t_1}^{t_0, \mu_\ep,  P_\ep,v}, 
\mu_{t_1}^{t_0, \mu_0,  P_0,v}) }\\
 \leq & W(t_0,t_1,\mu_0)+C\ep \;  \leq \; W(t_0,t_1,\mu_\ep)+C\ep \;.
\end{array}
$$

For $v\in \Vt(t_0)$ and $t\in [t_0,T]$, let  $f^{v,\ep}_t$ be the density of the measure
$\mu_{t}^{t_0,\mu_\epsilon, P_\epsilon,v}$.
We denote by $\f$ the closure in $L^1(\R^N)$ of the set $\{f^{v,\ep}_t\;,\; v\in \Vt(t_0),\; t\in [t_0,T]\}$. 
Since, from Lemma \ref{Lemme approx mu alpha}, the elements of $\f$ have a support contained in a fixed compact set $K_1$ and 
are uniformly bounded in ${\mathcal C}^1$, $\f$ is a compact subset of 
$L^1(\R^N)$. Therefore, for any fixed $\eta>0$, we can find a  partition $(O_i)_{i=1, \dots, n}$ of $\f$ 
into Borel (for the $L^1-$topology) subsets with a diameter in $L^1$ less than $\eta$. Let $f_i\in O_i$, $\mu_i=f_idx$ and
$ P_i\in\Delta( \A_x(t_1))$ be an $(\ep/6)-$optimal strategy for $\V^+(t_1, \mu_i)$. 
Let us check that, if $\eta$ is small enough, then the strategy $P_i$ is still $\ep/2-$optimal for $\V^+(t_1,\mu)$ for 
any measure $\mu\in \f$ such that $\|h_\mu-f_i\|_1\leq \eta$, where $h_\mu$ is the density of $\mu$. Indeed, for all $v\in\Vt(t_1)$,  we have
$$
\begin{array}{l}
 |\I(t_1,\mu, P_i,v)-\I(t_1,\mu_i,P_i,v)|\\
 \qquad \qquad \leq \; \ds{ \int_{\R^N\times \U(t_0)}\left|g(X_T^{t_1,x, u,v})(f_i(x)-h_\mu(x))
 \right| dP^v_{i,x}( u ) dx }\\
\qquad \qquad \leq \; \|g\|_\infty\|h_\mu-f_i\|_{L^1}\leq \eta\|g\|_\infty\;\leq \; \ep/6\;.
\end{array}
$$
So 
$$
\begin{array}{rl}
\ds{ \sup_{v\in  \Vt(t_1)}  \I(t_1,\mu,P_i,v) \;}
\leq  &\ds{ \sup_{v\in  \Vt(t_0)}  \I(t_1,\mu_i, P_i,v)+\ep/6 }\\
\leq & \V^+(t_1, \mu_i)+\ep/3\; \leq \; \V^+(t_1, \mu)+\ep/2\;.
\end{array}
$$

Let $\tau$ be a common delay for $ P_\ep$ and for all the $P_i$ ($i=1, \dots, I$). 
%We also assume that $\tau$ is so small that
%\be\label{dist ft1-delta ft1}
%\left\|f^{v,\ep}_{t_1-\tau}-f^{v,\ep}_{t_1}\right\|_{L^1(K_1)}\leq \frac{\eta}{2} \qquad\forall v\in \Vt(t_0)\;.
%\ee
For $v\in\Vt(t_0)$, we set $\mu_1^v=\mu_{t_1-\tau}^{t_0,\mu_\ep,  P_\ep,v}$. Since
$\U(t_0)=\U(t_0,t_1)\times \U(t_1)$, we can write any $u\in \U(t_0)$ as $u=(u_1,u_2)$ where $u_1\in \U(t_0,t_1)$
and $u_2\in \U(t_1)$. We define the strategy $P\in \Delta(\A(t_0))$ by 
$$
\int_{\U(t_0)} \varphi(u)dP^{ v}_x(u)
= \sum_{i=1}^I {\bf 1}_{\mu_1^v\in O_i} \int_{\U(t_0,t_1)\times \U(t_1)}\varphi((u_1,u_2))
 dP^{v_{|_{[t_1,T]}}}_{i,  X^{t_0,x, u_1,v}_{t_1-\tau}}(u _2) dP^v_{\ep,x}(u_1)
$$
(where, with a slight abuse of notation, $P^v_{\ep,x}$ still denotes the natural restriction of the measure $P^v_{\ep,x}$ to 
$\U(t_0,t_1)$) for any $(x,v)\in \R^N\times \Vt(t_0)$ and any nonnegative Borel measurable map $\varphi:\U(t_0)\to \R$. 
Then 
$$
\begin{array}{l}
\I(t_0,\mu_\ep, P, v) \\
  \ds{ = \sum_{i=1}^I {\bf 1}_{\mu_1^v\in O_i} \int_{\R^N\times \U(t_0,t_1)\times \U(t_1)} 
g\left( X_T^{t_1- \tau, X_{t_1-\tau}^{t_0,x, u_1 ,v},  (u_{1|_{[t_1-\tau,t_1]}},u_2) ,v_{_{|_{[t_1-\tau,T]}}}}\right)
dP^{v_{|_{[t_1,T]}}}_{i,  X^{t_0,x, u_1,v}_{t_1-\tau}}(u _2) dP^v_{\ep,x}(u_1)d\mu_\ep(x)   } \\
  \ds{ \leq \sum_{i=1}^I {\bf 1}_{\mu_1^v\in O_i} \int_{\R^N\times \U(t_0,t_1)\times \U(t_1)} 
\left[ g\left( X_T^{t_1, X_{t_1-\tau}^{t_0,x, u_1 ,v},  u_2 ,v_{_{|_{[t_1,T]}}}}\right)+C\tau\right]
dP^{v_{|_{[t_1,T]}}}_{i,  X^{t_0,x, u_1,v}_{t_1-\tau}}(u _2) dP^v_{\ep,x}(u_1)d\mu_\ep(x)   } \\
%\qquad +C\tau \\
% \ds{ \int_{\R^N\times \U(t_0,t_1)\times \U(t_1)} 
%\left| X_T^{t_1, X_{t_1-\tau}^{t_0,x, u_1 ,v},  u_2 ,v_{_{|_{[t_1,T]}}}}- 
%X_T^{t_1- \tau, X_{t_1-\tau}^{t_0,x, u_1 ,v},  (u_{1|_{[t_1-\tau,t_1]}},u_2) ,v_{_{|_{[t_1-\tau,T]}}}}\right|
%dP^{v_{|_{[t_1,T]}}}_{i,  X^{t_0,x, u_1,v}_{t_1-\tau}}(u _2) dP^v_{\ep}(x, u_1)   } \\
  \ds{ \leq  \sum_{i=1}^I {\bf 1}_{\mu_1^v\in O_i} \int_{\R^N\times \U(t_1)} 
g\left( X_T^{t_1, y,  u_2, v_{_{|_{[t_1,T]}}}}\right)
dP^{v_{|_{[t_1,T]}}}_{i, y}(u _2) d\mu_1^v(y)+C\tau } \\
%dP^{v_{|_{[t_1,T]}}}_{i,  y}(u _2)d\mu_{t_1}^{t_0,\mu_\ep, P_\ep,v}(y)   } \\
= \ds{ \sum_{i=1}^I {\bf 1}_{\mu_1^v\in O_i} \I(t_1,\mu_1^v, P_i,v_{|_{[t_1,T]}})+C\tau }
\end{array}
$$
Note that, if $\mu_1^v\in O_i$, then $\|f^{v,\ep}_{t_1-\tau}-f_i\|_{L^1}\leq \eta$, so that, by the choice
of $\eta$, we get
$$
\I(t_1,\mu_1^v,   P_i ,v_{|_{[t_1,T]}})
\leq \V^+(t_1,  \mu_1^v)+\ep/2\;.
$$
Therefore, recalling the definition of $\mu_1^v$ and noticing that $\d(\mu_1^v,\mu_{t_1}^{t_0,\mu_\ep, P_\ep,v})\leq C\tau$ , we get
$$
\I(t_0,\mu_\ep, P, v) \leq \V^+(t_1,  \mu_1^v)+\ep+C\tau
\leq \V^+(t_1,  \mu_{t_1}^{t_0,\mu_\ep, P_\ep,v})+\ep+C\tau\;.
$$
Now, since $ P_\ep$ is $C\ep-$optimal for $W(t_0,t_1,\mu_\ep)$ we obtain
$$
\V^+(t_0,\mu_\ep)\leq \sup_{v\in \Vt(t_0)}\I(t_0,\mu_\ep, P, v) \leq W(t_0,t_1, \mu_\ep)+C(\ep+\tau)\;.
$$
Using again the fact that $\V^+$ and $W$ are Lipschitz continuous we have, as
 $\epsilon$ and $\tau$ are is arbitrary:
\[\V^+(t_0,\mu_0)\leq W(t_0,t_1,\mu_0)\;.\]

Now we have to prove that the result still holds for measures with unbounded support.
Let $\mu_0\in\W$. For all $\epsilon>0$, there exists some closed ball $K_\epsilon$ centered at $0$ such that
$$
\int_{\R^N\backslash K_\epsilon}|x|^2d\mu_0(x)\leq \epsilon^2\qquad {\rm and}\qquad 
\mu_0(\R^N\backslash K_\epsilon)\leq \epsilon^2\;.
$$
Let $T:\R^N\to\R^N$ be such that $T(x)=x$ for all $x\in K_\epsilon$ and $T(x)=0$ for all $x\notin K_\epsilon$ and let us set
$\mu_\epsilon=T\sharp\mu_0$. Then 
$\mu_\epsilon\in\W(K_\epsilon)$ and $\d(\mu_0,\mu_\epsilon)\leq \epsilon$. 
Let us now check that for all $(P,v)\in\Delta( \A_x(t_0))\times\Vt(t_0)$, 
$\mu_1:=\mu_{t_1}^{t_0,\mu_0, P,v}$ is close to $\mu_1^\epsilon:=\mu_{t_1}^{t_0,\mu_\epsilon, P,v}$.
Indeed we have:
\ben
\begin{split}
\d^2(\mu_1,\mu_1^\epsilon)&\leq \int_{\R^N}\left|\int_{\U(t_0)}
X_{t_1}^{t_0,x,u,v}dP^v_x(u)-\int_{\U(t_0)}
X_{t_1}^{t_0,T(x),u,v}dP^v_{T(x)}(u)\right|^2d\mu_0(x)\\
&\leq \int_{\R^N\backslash K_\epsilon}\left|\int_{\U(t_0)}
X_{t_1}^{t_0,x,u,v}dP^v_x(u)-\int_{\U(t_0)}
X_{t_1}^{t_0,0,u,v}dP^v_{0}(u)\right|^2d\mu_0(x)\\
&\leq \int_{\R^N\backslash K_\epsilon}2\left[|x|^2+4(t_1-t_0)^2\|f\|_\infty^2\right]d\mu_0(x) \; \leq \; C\epsilon^2
\end{split}
\een
Then the Lipschitz continuity of the upper value leads to:
\ben
\begin{split}
\V^+(t_0,\mu_0)&\leq \V^+(t_0,\mu_\epsilon)+C\epsilon\\
&\leq \inf_{ P\in \Delta( \A_x(t_0))}\sup_{v \in\Vt(t_0)}\V^+(t_1,\mu_{t_1}^{t_0,\mu_\epsilon, P,v})+C\epsilon\\
&\leq  \inf_{ P\in \Delta( \A_x(t_0))}\sup_{v \in\Vt(t_0)}\V^+(t_1,\mu_{t_1}^{t_0,\mu_0, P,v})+C\epsilon\;.
\end{split}
\een
Hence $\V^+(t_0,\mu_0)\leq W(t_0,t_1,\mu_0)$ as $\epsilon$ is arbitrary.\\

We now prove that
\be\label{V+leqW}
\V^+(t_0,\mu_0)\geq W(t_0,t_1,\mu_0)\;.
\ee
Let $P$ be an $\epsilon$-optimal strategy for player I for $\V^+(t_0,\mu_0)$.
Let us fix $v_0\in\Vt(t_0)$ and set $\mu_1=\mu_{t_1}^{t_0,\mu_0,P,v_0}$. For all $v\in\Vt(t_1)$, we define the measure $\tilde P^v$ on 
$\R^N\times \U(t_1)$ by 
$$
\int_{\R^N\times \U(t_1)}\varphi(x, u_2)d\tilde P^v(x,u_2 )= \int_{\R^N\times \U(t_0)}
\varphi(X_{t_1}^{t_0,x,u,v_0}, u_{|_{[t_1,T]}}) dP^{(v_0|_{[t_0,t_1]},v)}_{x}(u)d\mu_0(x)
$$
for any $v\in  \Vt(t_1)$
and any nonnegative Borel measurable function $\varphi:\R^N\times \U(t_1)\to \R$. We note that the first  marginal of $\tilde P^v$ is
$\mu_{t_1}^{t_0,\mu_0, P,v_0}$. Since $\R^N\times \U(t_1)$ is a Polish space, we can desintegrate $\tilde P^v$ with respect to 
$\mu_{t_1}^{t_0,\mu_0, P,v_0}$: $d\tilde P^v(x,u)= d\tilde P^v_x(u)d\mu_{t_1}^{t_0,\mu_0, P,v_0}(x)$, where the mapping
$(x,v)\to \tilde P^v_x$ is measurable.
Then $\tilde P$ belongs to $\Delta( \A_x(t_1))$  and we have:
\begin{equation*}
\begin{split}
\V^+(t_1,\mu_{t_1}^{t_0,\mu_0, P,v_0})&\quad\leq \sup_{v\in\Vt(t_1)}\int_{\R^{N}\times \U(t_1)} 
g(X_{T}^{t_1,x, u_2,v}) d\tilde P_{x}^{v}(u_2)d\mu_{t_1}^{t_0,\mu_0, P,v_0}(x)\\
&\quad\leq\sup_{v\in\Vt(t_1)}\int_{\R^N\times\U(t_0,t_1)\times \U(t_1)}g\left(X_{T}^{t_1,X_{t_1}^{t_0,x, u_1,v_0},
u_2 ,v}\right)dP^{ (v_0|_{[t_0,t_1]},v)}_x((u_1, u_2 ))d\mu_0(x) \\
%&\quad\leq\sup_{v\in\Vt(t_1)}\int_{\R^N\times\U(t_0)}g(X_{T}^{t_1,X_{t_1}^{t_0,x,\alpha_\epsilon(x, u ),v_0},\alpha_\epsilon(x, u ),v})dm_\epsilon(x, u )\\
&\quad\leq\sup_{v\in\Vt(t_1)}\int_{\R^N\times \U(t_0)}
 g\left(X_{T}^{t_0,x, u,(v_0|_{[t_0,t_1]},v)}\right) dP^{ (v_0|_{[t_0,t_1]},v)}_x( u )d\mu_0(x)\\
&\quad\leq\V^+(t_0,\mu_0)+\epsilon
\end{split}
\end{equation*}
Hence  inequality (\ref{V+leqW}) holds since $v_0$ and $\epsilon$ are arbitrary.
\end{proof}

%%%%%%%%%%%%%%%%%%%%%%%%%%%%%%%%%%%%%%%%%%%%%%%%%%%%%%%%
\section{Characterization of the upper value function}
%%%%%%%%%%%%%%%%%%%%%%%%%%%%%%%%%%%%%%%%%%%%%%%%%%%%%%%%
\label{Caracterisation de la valeur solution HJ}
We prove in this section that if a function satisfies the previous dynamic programming principle, then it is the unique viscosity solution of some Hamilton-Jacobi equation.\\

We consider the Hamiltonian $H$, defined for any $\mu\in \W$ and for any $\p\in L^2_{\mu}(\R^N,\R^N)$, by
\be\label{DefH}
H(\mu,\p)= \sup_{{\bf v}\in \Delta(V)} \int_{\R^N} \inf_{{\bf u}\in \Delta (U)} 
\int_{U\times V}\lg f(x,u,v),\p(x)\rg d{\bf u}(u)d{\bf v}(v) d\mu(x)\;,
\ee
where $\Delta(U)$ and $\Delta(V)$ denote the sets of Borel probability measures on the compact sets 
$U$ and $V$ respectively. 
Let $\V: [0,T]\times \W\to \R$ be a Lipschitz continuous map. We say $\V$ is a subsolution to
\be\label{HJ}
%\left\{\begin{array}{l}
\V_t+H(\mu, D_\mu \V)=0 \qquad {\rm in }\; [0,T)\times \W%\\
%\V(T,\mu)=\int_{\R^N} g(x)d\mu(x)
%\end{array}\right.
\ee
if, for any test function $\varphi(t,\mu)$ of the form
$$
\varphi(t,\mu)=\frac{\alpha}{2} \d^2(\bmu,\mu)+\eta \d(\bnu,\mu)+\psi(t)
$$
(where $\psi:\R\to\R$ is smooth, $\alpha,\eta>0$ and $\bmu,\bnu\in \W$)  such that $\V-\varphi$ has a local
maximum at $(\bnu, \bt)\in[0,T)\times\W$ and  for any optimal transport plan
$\bar \pi \in \Pi_{opt}(\bmu,\bnu)$, one has
$$
\psi'(\bt)+H(\bnu, -\alpha \p)\geq -\|f\|_\infty\eta
$$
where $\p$ is the unique element of $L^2_\bnu(\R^N,\R^N)$ associated to $\bar \pi$ such that
\be\label{def bf p}
\int_{\R^N} \lg \xi(y),x-y\rg d\bar \pi(x,y)=\int_{\R^{N}} \lg \xi(y),\p(y)\rg d\bnu(y)\qquad \forall \xi\in L^2_{\bnu}(\R^N,\R^N)
\ee
(see \cite{quincampoix}). 
In the same way, we say $\V$ is a supersolution to \eqref{HJ}
if, for any test function $\varphi(t,\mu)$ of the form
$$
\varphi(t,\mu)=-\frac{\alpha}{2} \d^2(\bmu,\mu)-\eta \d(\bnu,\mu)+\psi(t)
$$
(where $\psi:\R\to\R$ is smooth, $\alpha,\eta>0$ and $\bmu,\bnu\in \W$) such that $\V-\varphi$ has a local
minimum at $(\bnu, \bt)\in[0,T)\times\W$, one has
$$
\psi'(\bt)+H(\bnu, \alpha \p)\leq \|f\|_\infty\eta \;.
$$

\begin{Proposition}[Comparison principle]\label{comparaison} Let $w_1$ be a Lipschitz continuous subsolution of \eqref{HJ} and $w_2$ be a Lipschitz continuous supersolution
such that $w_1(T, \mu)\leq w_2(T,\mu)$ for any $\mu\in \W$. Then $w_1\leq w_2$ in $[0,T]\times \W$.  
\end{Proposition}

In particular, given a Lipschitz continuous terminal condition $\tilde g:\W\to \R$,
 the Hamilton-Jacobi equation \eqref{HJ} has at most one Lipschitz continuous solution $\V$ which satisfies
$\V(T,\mu)=\tilde g(\mu)$ for any $\mu\in \W$.

Note that other definitions of viscosity solution in the space $\W$ have been  introducted recently: see for instance
\cite{FK09, FS10, GNT, LcoursColl}. Our definition is closely related to the one of \cite{quincampoix}, which seems more appropriate 
for the kind of problem we have to handle.

\begin{proof}
The proof borrows its main arguments from \cite{Crandall}, and follows closely \cite{quincampoix}.
We denote  by $K$ the common Lipschitz constant of $w_1$, $w_2$ and $f$.
Without loss of generality we can assume that
\be
\label{egalite w1 w2 en T}
\inf_{\mu\in\W}w_2(T,\mu)-w_1(T,\mu)=0\;.
\ee
Our aim is to prove that \[\inf_{(t,\mu)\in[0,T]\times\W}w_2(t,\mu)-w_1(t,\mu)=0\;.\]
Assume on the contrary that \[\inf_{(t,\mu)\in[0,T]\times\W}w_2(t,\mu)-w_1(t,\mu)=-\xi<0\;.\]
Fix $(t_0,\mu_0)$ such that $(w_2-w_1)(t_0,\mu_0)<-\xi/2$.
Denote by \[\varphi_{\epsilon\eta}(s,\mu,t,\nu)=w_2(t,\nu)-w_1(s,\mu)+\frac{1}{\epsilon}\d^2(\mu,\nu)+\frac{1}{\epsilon}(t-s)^2-\eta s\;.\]
The function $\varphi_{\epsilon\eta}$ is continuous and bounded from below.
Using some modified version of Ekeland's variational Lemma (Lemma \ref{Ekeland} below), we have that, 
for all $\delta>0$, there is $(\bs,\bmu,\bt,\bnu)$ such that for all $(s,\mu,t,\nu)$:
\be
\label{Ekeland sur Phi}
\begin{array}{ll}
\varphi_{\epsilon\eta}(\bs,\bmu,\bt,\bnu)\leq \varphi_{\epsilon\eta}(t_0,\mu_0,t_0,\mu_0)\\
\varphi_{\epsilon\eta}(\bs,\bmu,\bt,\bnu)\leq \varphi_{\epsilon\eta}(s,\mu,t,\nu)+\delta[ \d(\mu,\bmu)+\d(\nu,\bnu)]
\end{array}
\ee
Let us fix $\pi\in\Pi_{opt}(\bnu,\bmu)$. 
We first give a bound on the distance between $(\bs,\bmu)$ and $(\bt,\bnu)$. Since
\[\varphi_{\epsilon\eta}(\bs,\bmu,\bt,\bnu)\leq \varphi_{\epsilon\eta}(\bs,\bmu,\bs,\bmu)+\delta \d(\bmu,\bnu)\;, \]
we have
\ben
\begin{split}
&w_2(\bs,\bmu)-w_1(\bs,\bmu)-\eta \bs+\delta \d(\bmu,\bnu)\\
&\quad\geq w_2(\bt,\bnu)-w_1(\bs,\bmu)+\frac{1}{\epsilon}\d^2(\bmu,\bnu)+\frac{1}{\epsilon}(\bt-\bs)^2-\eta \bs\\
&\quad\geq w_2(\bs,\bmu)- K |\bs-\bt|-K d(\bnu,\bmu)-w_1(\bs,\bmu)+\frac{1}{\epsilon}\d^2(\bmu,\bnu)+\frac{1}{\epsilon}(\bt-\bs)^2-\eta \bs\;,
\end{split}
\een
that reduces to% \[\epsilon (K+\delta)\d(\bnu,\bmu)+\epsilon K |\bt-\bs|\geq \d^2(\bnu,\bmu)+(\bt-\bs)^2\geq \frac{\left(\d(\bnu,\bmu)+|\bt-\bs|\right)^2}{2}\]
%and finally
\be
\label{estimation dmunu unicité}
\d(\bnu,\bmu)+|\bt-\bs|\leq 2\epsilon(K+\delta)\;.
\ee
%for $\delta,\epsilon$ small enough.\\
We now seek some contradiction assuming that $\bt,\bs\neq T$.
We first use the fact that
\[\varphi_{\epsilon\eta}(\bs,\bmu,\bt,\bnu)\leq \varphi_{\epsilon\eta}(s,\mu,\bt,\bnu)+\delta\d(\mu,\bmu)\;, \]
namely
\ben
\begin{split}
&w_2(\bt,\bnu)-w_1(\bs,\bmu)+\frac{1}{\epsilon}\d^2(\bmu,\bnu)+\frac{1}{\epsilon}(\bt-\bs)^2-\eta \bs\\
&\leq w_2(\bt,\bnu)-w_1(s,\mu)+\frac{1}{\epsilon}\d^2(\mu,\bnu)+\frac{1}{\epsilon}(\bt-s)^2-\eta s +\delta\d(\mu,\bmu)\;,
\end{split}
\een
leading to
\ben
w_1(s,\mu)-\frac{1}{\epsilon}\d^2(\mu,\bnu)-\delta\d(\mu,\bmu)-\frac{1}{\epsilon}(\bt-s)^2+\eta s \\
\leq w_1(\bs,\bmu)-\frac{1}{\epsilon}\d^2(\bmu,\bnu)-\frac{1}{\epsilon}(\bt-\bs)^2+\eta \bs\;.
\een
If we set $\varphi(s,\mu)=\frac{1}{\epsilon}\d^2(\mu,\bnu)+\delta\d(\mu,\bmu)+\frac{1}{\epsilon}(\bt-s)^2-\eta s$,
then the function $w_1-\varphi$ has a maximum at $(\bs,\bmu)$.
The function $w_1$ being a subsolution, we get by definition:
\be\label{(A)}
-\eta -\frac{2}{\epsilon}(\bt-\bs)+H(\bmu,-\frac{2}{\epsilon}\p)\geq -\delta\|f\|_\infty
\ee
where $\p$ is defined by:
\[\int_{\R^{2N}}\langle\xi(y),x-y\rangle d\pi(x,y)=\int_{\R^N}\langle\xi(y),\p(y)\rangle d\bmu(y) \qquad \forall \xi\in L^2_\bmu(\R^N,\R^N)\;.\]
The same argument applied to
\[\varphi_{\epsilon\eta}(\bs,\bmu,\bt,\bnu)\leq \varphi_{\epsilon\eta}(\bs,\bmu,t,\nu)+\delta\d(\nu,\bnu)\]
leads to
%\ben
%\begin{split}
%&w_2(\bt,\bnu)-w_1(\bs,\bmu)+\frac{1}{\epsilon}\d^2(\bmu,\bnu)+\frac{1}{\epsilon}(\bt-\bs)^2-\eta \bs\\
%&\leq w_2(t,\nu)-w_1(\bs,\bmu)+\frac{1}{\epsilon}\d^2(\bmu,\nu)+\frac{1}{\epsilon}(t-\bs)^2-\eta \bs+\delta\d(\nu,\bnu)
%\end{split}
%\een
%This means that denoting by $\varphi(t,\nu)=-\frac{1}{\epsilon}\d^2(\bmu,\nu)-\delta\d(\nu,\bnu)-\frac{1}{\epsilon}(t-\bs)^2$,
%the function $w_2-\varphi$ has a minimum at $(\bt,\bnu)$.
%The function $w_2$ being a supersolution, we get by definition:
\be\label{(B)}
-\frac{2}{\epsilon}(\bt-\bs)+H(\bnu,\frac{2}{\epsilon}\q)\leq \delta\|f\|_\infty
\ee
where $\q$ satisfies
\[\int_{\R^{2N}}\langle\xi(y),x-y\rangle d\bar\pi(x,y)=\int_{\R^N}\langle\xi(x),\q(x)\rangle d\bnu(x) \qquad \forall \xi\in L^2_\bnu(\R^N,\R^N)\;,\]
$\bar\pi$ being defined by
\[\int_{\R^{2N}}\varphi(x,y) d\bar\pi(x,y)=\int_{\R^N}\varphi(y,x)d\pi(x,y) \ \forall \varphi\in L^2_\pi(\R^{2N},\R^{2N})\;.\]
Note that 
\[\int_{\R^{2N}}\langle\xi(x),x-y\rangle d\pi(x,y)=\int_{\R^N}\langle\xi(x),-\q(x)\rangle d\bnu(y) \qquad \forall \xi\in L^2_\bnu(\R^N,\R^N)\;.\]
Combining (\ref{(A)}) and (\ref{(B)}) we get
\be
\label{eq intermediaire H}
\eta +H(\bnu,\frac{2}{\epsilon}\q)-H(\bmu,-\frac{2}{\epsilon}\p)\leq 2\delta\|f\|_\infty\;.
\ee
Let us now recall some continuity property of the Hamiltonian $H$ defined by (\ref{DefH}):
\begin{Lemma}
\label{continuity of the hamiltonian}
Let $(\bnu,\bmu)\in\W^2$ and $(\p,\q)\in L^2_\bmu(\R^N,\R^N)\times L^2_\bnu(\R^N,\R^N)$ be such that,
for some $\pi\in\Pi_{opt}(\bnu,\bmu)$,
\[\int_{\R^{2N}}\langle\xi(y),x-y\rangle d\pi(x,y)=\int_{\R^N}\langle\xi(y),\p(y)\rangle d\bmu(y) \qquad \forall \xi\in L^2_\bmu(\R^N,\R^N)\]
and
\[\int_{\R^{2N}}\langle\xi(x),x-y\rangle d\pi(x,y)=\int_{\R^N}\langle\xi(x),-\q(x)\rangle d\bnu(x) \qquad \forall \xi\in L^2_\bnu(\R^N,\R^N)\;.\]
Then we have:
\[|H(\bmu,\p)-H(\bnu,-\q)|\leq K\d^2(\bnu,\bmu)\]
where $K$ stands for the Lipschitz constant of the dynamics.
\end{Lemma}
\begin{proof}
The proof is the same as in \cite{quincampoix}, Lemma 6. 
\end{proof}
Therefore, we have:
\[|H(\bmu,-\frac{2}{\epsilon}\p)-H(\bnu,\frac{2}{\epsilon}\q_y)|\leq \frac{2K}{\epsilon}\d^2(\bnu,\bmu)\;.\]
Thus using the previous inequality and estimate \eqref{estimation dmunu unicité} in \eqref{eq intermediaire H}, we get
\[ \eta\leq 2\delta\|f\|_\infty+ 8K\epsilon(K+\delta)^2\]
leading to a contradiction for $\epsilon,\delta$ sufficiently small.\\

This implies that we have $\bt=T$ or $\bs=T$.
Assume for example that $\bs=T$.
We have \[\varphi_{\epsilon\eta}(\bs,\bmu,\bt,\bnu)\leq \varphi_{\epsilon\eta}(t_0,\mu_0,t_0,\mu_0)\leq -\xi/2\;.\]
Therefore using \eqref{egalite w1 w2 en T} and \eqref{Ekeland sur Phi} we obtain:
\ben
\begin{split}
-\xi/2&\geq w_2(\bt,\bnu)-w_1(T,\bmu)+\frac{1}{\epsilon}\d^2(\bmu,\bnu)+\frac{1}{\epsilon}(T-\bt)^2-\eta T\\
&\geq -K|T-\bt|-K\d(\bnu,\bmu)+\frac{1}{\epsilon}\d^2(\bmu,\bnu)+\frac{1}{\epsilon}(T-\bt)^2-\eta T\\
&\geq -K[|T-\bt|+\d(\bmu,\bnu)]+\frac{1}{2\epsilon}[|T-\bt|+\d(\bmu,\bnu)]^2-\eta T
\end{split}
\een
Using \eqref{estimation dmunu unicité}, we finally get:
\[\xi/2\leq 2\epsilon(K+\delta)(2K+\delta)+\eta T\]
which is impossible for $\epsilon$ and $\eta$ small enough.
\end{proof}

\begin{Proposition}\label{V+ visc sol} The upper value function $\V^+$ is the unique Lipschitz continuous 
viscosity solution of the Hamilton-Jacobi equation \eqref{HJ} satisfying the terminal condition:
\be
\label{condition finale}
\V^+(T,\mu)=\int_{\R^N} g(x)d\mu(x)
\ee
\end{Proposition}

\begin{proof}
We only prove that $\V^+$ is some solution, uniqueness being an obvious consequence of Proposition \ref{comparaison}. 
Let us recall that $\V^+$ satisfies the dynamic programming principle
\be\label{dynamo}
\V^+(t_0,\bnu)=\inf_{ P\in \Delta( \A_x(t_0))}\sup_{v \in\Vt(t_0)}\V^+(t_0+h,\mu_{t_0+h}^{t_0,\bnu, P,v})\;
\ee
where $\mu_{t_0+h}^{t_0,\bnu, P,v}$ is the measure defined by
\[\forall \varphi\in \C_b(\R^N,\R),\ \int_{\R^N} \varphi(x)d\mu_{t_0+h}^{t_0,\bnu, P,v}(x)=\int_{\R^N\times\U(t_0)} \varphi(X_{t_0+h}^{t_0,x, u,v})d
P^{v}(x, u )\;.\]
%
%
%$$
%\int_{\R^N}\varphi(x)d\nu(t_0+h)(x)=\int_{\R^N}{\bf E}_{ P(x)}\left[\varphi(X_{t_0+h}^{t_0,x, P(x),v})\right]d\bnu(x)
%\qquad \forall \varphi\in {\mathcal C}_b(\R^N,\R)\;.
%$$
Let us show that $\V^+$ is a subsolution. Let  $\varphi=\varphi(t,\mu)$ be a  test function of the form
$$
\varphi(t,\mu)=\frac{\alpha}{2} \d^2(\bmu,\mu)+\eta \d(\bnu,\mu)+\psi(t)
$$
(where $\psi:\R\to\R$ is smooth, $\alpha,\eta>0$ and $\bmu,\bnu\in \W$), such that $\V^+-\varphi$ has a local
maximum at $(\bnu, t_0)$. Without loss of generality we assume that $\varphi(t_0,\bnu)=\V^+(t_0,\bnu)$.
Then $\V^+\leq \varphi$. We fix an optimal plan $\bar \pi\in \Pi_{opt}(\bmu,\bnu)$.
From (\ref{dynamo}), we get
\be\label{pgrphi}
0\leq \inf_{ P\in \Delta( \A_x(t_0))}\sup_{v \in\Vt(t_0)}\left[\varphi(t_0+h,\mu_{t_0+h}^{t_0,\bnu, P,v})-\varphi(t_0,\bnu)\right]\;.
\ee
%$$
%0\leq \inf_{ P\in {\mathcal A}_x(t_0)}\sup_{v\in {\mathcal V}(t_0)} \left[ \varphi(t_0+h,\nu(t_0+h))-\varphi(t_0,\bnu)\right]
%$$
Setting for simplicity $\nu(t_0+h)=\mu_{t_0+h}^{t_0,\bnu, P,v}$ and recalling the definition of $\varphi$ we have
\be\label{pgrphi2}
\varphi(t_0+h,\nu(t_0+h))-\varphi(t_0,\bnu)
=
\frac{\alpha}{2} \left[\d^2(\bmu,\nu(t_0+h))-\d^2(\bmu,\bnu)\right]+\eta \d(\bnu,\nu(t_0+h))+\psi(t_0+h)-\psi(t_0)
\ee
where
$$
\d(\bnu,\nu(t_0+h)) \leq \left[\int_{\R^{N}\times \U(t_0)} \left|y-X_{t_0+h}^{t_0, y,u,v}\right|^2dP^{ v}_y( u )d\bnu(y)\right]^{\frac12}
\leq \|f\|_\infty h\;.
$$
Recalling the definition of ${\bf p}$ in (\ref{def bf p}), we also have
\ben
\begin{split}
&\d^2(\bmu,\nu(t_0+h)) \\
&\quad \leq  \int_{\R^{2N}\times \U(t_0)} \left|x-X_{t_0+h}^{t_0, y, u,v}\right|^2dP^{ v}_y( u )d\bar \pi(x,y) \\
%&\quad\leq \ \ds{ \d^2(\bmu,\bnu)+ 2\int_{\R^{2N}} {\bf E}_{ P(y)}\left[\lg x-y, y-X_{t_0+h}^{t_0, y, P(y)(v),v}\rg\right] d\bar \pi(x,y)+ Ch^2 }\\
&\quad\leq \d^2(\bmu,\bnu)- 2\int_{\R^{2N}} \lg x-y,
\left[\int_{\U(t_0)}\int_{t_0}^{t_0+h} f(X_s^{t_0, y, u,v}, u(s),v(s))ds\ dP^{v}_y( u )
\right]\rg d\bar \pi(x,y)+ Ch^2 \\
&\quad \leq \d^2(\bmu,\bnu)- 2\int_{\R^{N}} \lg \p(y),
\left[\int_{\U(t_0)} \int_{t_0}^{t_0+h} f(y, u(s),v(s))ds\ dP^{v}_y( u )\right]\rg d\bnu(y)+ Ch^2
\end{split}
\een
Let $(U_n)$ be an increasing family of  finite subsets of $U$ such that $\bigcup_n U_n$ is dense in $U$ and 
${\bf U}^n$ be the set of Borel measurable maps $x\to {\bf u}_x$ from $\R^N$ into $\Delta(U_n)$.
The main point in this discretization is that $\Delta(U_n)$ is a compact subset of some finite dimensional space.
Therefore ${\bf U}^n$, endowed with the weak topology of $L^2_{\bar \nu}$, is convex and compact.
Since $U_n$ can be viewed as a subset of $\U(t_0)$, one can associate with a map ${\bf u}\in {\bf U}^n$ a strategy 
$P_{{\bf u}}\in \Delta( \A_x(t_0)$) defined by the equality
$$
\int_{\R^N\times \U(t_0)}\varphi(x,u)dP^v_{{\bf u}} =\int_{\R^N\times U_n} \varphi(x,u)d{\bf u}_x(u)d\bnu(x)
$$
for any $v\in \Vt(t_0)$ and any nonnegative Borel measurable map $\varphi: \R^N\times \U(t_0)\to \R$. 
 Recalling (\ref{pgrphi}) and (\ref{pgrphi2}), we get
$$
\begin{array}{l}
0\leq \psi'(t_0)+ \|f\|_\infty \eta +Ch\\
\qquad \ds{ + \alpha\inf_{{\bf u}\in {\bf U}^n} \sup_{Q\in \Delta( {\mathcal V}(t_0))} \int_{\R^{N}} \lg -\p(y),
\int_{U_n\times \Vt(t_0)} \frac{1}{h}\int_{t_0}^{t_0+h} f(y,u,v(s))ds\ d{\bf u}_y(u) dQ(v) \rg d \bnu(y)\;.}
\end{array}
$$
By Sion's min-max Theorem we get 
$$
\begin{array}{l}
\ds{  \inf_{{\bf u}\in {\bf U}^n} \sup_{Q\in \Delta( {\mathcal V}(t_0))} \int_{\R^{N}} \lg -\p(y),
\int_{U_n\times \Vt(t_0)} \frac{1}{h}\int_{t_0}^{t_0+h} f(y,u,v(s))ds\ d{\bf u}_y(u) dQ(v) \rg d \bnu(y) }\\
%\qquad =
%\ds{ \inf_{{\bf u}\in {\bf U}^n}  \sup_{Q\in \Delta( {\mathcal V}(t_0))} \int_{\R^{N}} \lg -\p(y),
%\int_{U_n\times \Vt(t_0)} \frac{1}{h}\int_{t_0}^{t_0+h} f(y,u,v(s))ds\ d{\bf u}_y(u)dQ(v) \rg d \bnu(y)  }\\
\qquad =
\ds{ \sup_{Q\in \Delta(  {\mathcal V}(t_0))} \inf_{{\bf u}\in {\bf U}^n}   \int_{\R^{N}} \lg -\p(y),
\int_{U_n\times \Vt(t_0)} \frac{1}{h}\int_{t_0}^{t_0+h} f(y,u,v(s))ds\ d{\bf u}_y(u) dQ(v)\rg d \bnu(y) }\\
\qquad =
\ds{ \sup_{Q\in \Delta(  {\mathcal V}(t_0))}  \int_{\R^{N}} \inf_{{\bf u}\in \Delta(U_n)} \lg -\p(y), 
\int_{U_n\times  \Vt(t_0)} \frac{1}{h}\int_{t_0}^{t_0+h} f(y,u,v(s))ds\  d{\bf u}(u) dQ(v)\rg d \bnu(y) }\\
\qquad \leq
\ds{ \sup_{Q\in \Delta(  {\mathcal V}(t_0))}  \int_{\R^{N}} \inf_{{\bf u}\in \Delta(U_n)} \esssup_{s\in [t_0,t_0+h]} 
\lg -\p(y), \int_{U_n\times  \Vt(t_0)}  f(y,u,v(s))\  d{\bf u}(u) dQ(v)\rg d \bnu(y) }\\
\qquad \leq
\ds{ \sup_{{\bf v} \in \Delta(V) }  \int_{\R^{N}} \inf_{{\bf u}\in \Delta(U_n)} \lg -\p(y), 
\int_{U_n\times  V}  f(y,u,v)\  d{\bf u}(u) d{\bf v}(v)\rg d \bnu(y) }
\end{array}
$$
So
$$
0\; \leq \; \ds{  \psi'(t_0)+ \alpha \sup_{{\bf v} \in \Delta(V) }  \int_{\R^{N}} \inf_{{\bf u}\in \Delta(U_n)} \lg -\p(y), 
\int_{U_n\times  V}  f(y,u,v)\  d{\bf u}(u) d{\bf v}(v)\rg d \bnu(y)+ \|f\|_\infty \eta +Ch  }
$$
Letting $h\to 0$ and $n\to+\infty$ gives  the desired inequality since $\bigcup_n \Delta(U_n)$ is dense in $\Delta(U)$. \\

We now check that $\V^+$ is a supersolution. Let $\varphi=\varphi(t,\mu)$ be a  test function of the form
$$
\varphi(t,\mu)=-\frac{\alpha}{2} \d^2(\bmu,\mu)-\eta \d(\bnu,\mu)+\psi(t)
$$
(where $\psi:\R\to\R$ is smooth, $\alpha,\eta>0$ and $\bmu,\bnu\in \W$), such that $\V^+-\varphi$ has a local
minimum at $(\bnu, t_0)$. We again assume that $\varphi(t_0,\bnu)=\V^+(t_0,\bnu)$, so that
$\V^+\geq \varphi$.  Let us apply the dynamic programming at time $t_0$ and for $\bnu$. We get
$$
0\geq \inf_{ P\in \Delta( \A_x(t_0))}\sup_{v \in\Vt(t_0)}\left[\varphi(t_0+h,\mu_{t_0+h}^{t_0,\bnu, P,v})-\varphi(t_0,\bnu)\right]
$$
Setting as before $\bar \nu(t_0+h)=\mu_{t_0+h}^{t_0,\bnu, P,v}$, we have, for any $ P\in \Delta( {\mathcal A}_x(t_0))$ and $v\in {\mathcal V}(t_0)$,
$$
\d^2(\bmu,\nu(t_0+h))
\ \leq \ \d^2(\bmu,\bnu)- 2\int_{\R^{N}} \lg \p(y),
\left[\int_{\U(t_0)}\int_{t_0}^{t_0+h} f(y, u(s),v(s))ds\ dP^{v}_y( u )\right]\rg d\bnu(y)- Ch^2
$$
so that
\be\label{kjqneca}
\begin{array}{l}
0\; \geq \;  \psi'(t_0)- \|f\|_\infty \eta -Ch \\
\; \ds{  +\alpha \inf_{ P\in \Delta( {\mathcal A}_x(t_0))}\sup_{Q\in \Delta( {\mathcal V}(t_0))}
\int_{\R^{N}} \lg \p(y),
\left[\int_{\U(t_0)\times \Vt(t_0)}\frac{1}{h}\int_{t_0}^{t_0+h} f(y, u(s),v(s))ds\ dP^{v}_y( u )dQ(v)\right]
\rg d\bnu(y) } 
\end{array}
\ee
Let $P_h\in \Delta( \A_x(t_0))$ be $h-$optimal in the above expression. We denote by $\tau_h$ its delay and set $n_h=h/\tau_h$.
Reducing $\tau_h$ if necessary, we can suppose that $n_h$ is a positive integer. 
Let us set $t_k=t_0+k\tau_h$ for $k=0, \dots, n_h$. 
Let us now fix
${\bf v}\in \Delta(V)$. With ${\bf v}$ we associate the strategy $Q_{h, {\bf v}}$  consisting in choosing randomly, on each time interval
$[t_k, t_{k+1}]$ (where $k=0, \dots ,n_h-1$) a time independant control $v$  according
to the probability ${\bf v}$. 
%Since $P_h$ has a delay $\tau_h$, 
%a backward induction shows that 
We now claim that
\be\label{lkjezrf}
\begin{array}{l}
\ds{ \int_{\U(t_0)\times \Vt(t_0)} \lg \p(y),
\int_{t_0}^{t_0+h} f(y, u(s),v(s))ds\rg \ dP^{v}_{h,y}( u )dQ_{h, {\bf v}}(v)  }\\
\qquad\qquad \qquad \ds{  \geq h \inf_{{\bf u}\in \Delta(U)} 
\int_{U\times V} \lg \p(y),
f(y, u,v)\rg \ d{\bf u}( u )d{\bf v}(v)}
\end{array}
\ee
for $\bar \nu-$a.e. $y$. For this it is enough to show by backward induction on $k\in \{0,\dots, n_h\}$ that 
\be\label{iouyerz}
\begin{array}{l}
\ds{ \int_{\U(t_0)\times \Vt(t_0)} \lg \p(y),
\int_{t_0}^{t_0+h} f(y, u(s),v(s))ds\rg \ dP^{v}_{h,y}( u )dQ_{h, {\bf v}}(v)  }\\
\qquad \geq \ds{ \int_{\U(t_0,t_k)\times \Vt(t_0,t_k)} \lg \p(y),
\int_{t_0}^{t_k} f(y, u(s),v(s))ds\rg \ dP^{k,v}_{h,y}( u )dQ^k_{h, {\bf v}}(v)}\\
\qquad \qquad \ds{  +
(n_h-k)\tau_h \inf_{{\bf u}\in \Delta(U)} 
\int_{U\times V} \lg \p(y),
f(y, u,v)\rg \ d{\bf u}( u )d{\bf v}(v)}
\end{array}
\ee
where $P^k_h$ and $Q^k_{h,{\bf v}}$ are defined as the restriction of the strategies $P_h$ and $Q_{h,{\bf v}}$ to the time 
interval $[t_0,t_0+k\tau_h]$.
Note that the above inequality is obvious for $k=n_h$. Let us assume that it holds for $k+1$
and prove that it still hold for $k$.  We use the decomposition 
$$
\Vt(t_0,t_{k+1})=\Vt(t_0,t_k)\times \Vt(t_k,t_{k+1})
$$
and write $v=(v_1,v_2)$ for any $v\in \Vt(t_0,t_{k+1})$, where $v_1\in \Vt(t_0,t_k)$ and $v_2\in \Vt(t_k,t_{k+1})$.
By definition of $Q^{k+1}_{h, {\bf v}}$, we have 
$$
\int_{\Vt(t_0,t_{k+1})}  dQ^{k+1}_{h, {\bf v}}(v) =\int_{\Vt(t_0,t_k)\times V}  dQ^k_{h, {\bf v}}(v_1)d{\bf v}(v_2)
$$
Therefore, since $P$ has a delay $\tau_h$ we get,
%$$
%\int_{\U(t_0,t_{k+1})} \lg \p,
%\int_{t_0}^{t_{k+1}} fds\rg \ dP^{k+1,v}_{h,y}( u )
%=
%\int_{\U(t_0,t_{k+1})} \lg \p,
%\int_{t_0}^{t_{k+1}} fds\rg \ dP^{k+1,v_1}_{h,y}( u )\;,
%$$
%where for simplicity we 
omitting the arguments of $\p$ and $f$ for simplicity
$$
\begin{array}{l}
\ds{ \int_{\U(t_0,t_{k+1})\times \Vt(t_0,t_{k+1})} \lg \p, \int_{t_0}^{t_{k+1}} f ds\rg \ dP^{k+1,v}_{h,y}( u )dQ^{k+1}_{h, {\bf v}}(v) }\\
\qquad \ds{ = \int_{\U(t_0,t_{k+1})\times \Vt(t_0,t_{k})\times V} 
\lg \p, \int_{t_0}^{t_{k+1}} f ds\rg \ d{\bf v}(v_2) dP^{k+1,v_1}_{h,y}( u )dQ^{k}_{h, {\bf v}}(v_1) }\\
\qquad \ds{ \geq  \int_{\U(t_0,t_{k})\times \Vt(t_0,t_{k})} 
\lg \p, \int_{t_0}^{t_{k}} f ds\rg dP^{k,v_1}_{h,y}( u )dQ^{k}_{h, {\bf v}}(v_1) }\\
\qquad \qquad \ds{ + \tau_h\inf_{{\bf u}\in \Delta(U)} \int_{U\times V} \lg \p, f ds\rg d{\bf u}( u )d{\bf v}(v)
}
\end{array}
$$
because 
$$
\begin{array}{l}
\ds{ \int_{\U(t_0,t_{k+1})\times \Vt(t_0,t_{k})\times V} 
\lg \p, \int_{t_k}^{t_{k+1}} f(y,u(t),v_2) ds\rg \ d{\bf v}(v_2) dP^{k+1,v_1}_{h,y}( u )dQ^{k}_{h, {\bf v}}(v_1) }\\
\qquad \ds{ \geq \inf_{P\in \Delta( \U(t_k,t_{k+1}))} 
\int_{\U(t_k,t_{k+1})\times V} \lg \p, \int_{t_k}^{t_{k+1}} f(y,u(t),v_2) ds\rg \ d{\bf v}(v_2) dP(u) }\\
\qquad = \tau_h  \ds{ \inf_{{\bf u}\in \Delta(U)} 
\int_{U\times V} \lg \p, \int_{t_k}^{t_{k+1}} f(y,u,v_2) ds\rg \ d{\bf u}(u)d{\bf v}(v_2)  }
\end{array}
$$
where $\Delta( \U(t_k,t_{k+1}))$ stands for the Borel probability measures on the set $ \U(t_k,t_{k+1})$.
This gives (\ref{iouyerz}) by induction.

Combining (\ref{kjqneca}) with (\ref{lkjezrf}) we get
$$
\begin{array}{l}
0\; \geq \;  \psi'(t_0)- \|f\|_\infty \eta -Ch 
%\\
%\qquad 
\ds{  
+\alpha 
\int_{\R^{N}} \inf_{{\bf u}\in \Delta(U)} 
\int_{U\times V} \lg \p(y),
f(y, u,v)\rg \ d{\bf u}( u )d{\bf v}(v) d\bnu(y) } \;,
\end{array}
$$
and we obtain the desired inequality by letting $h\to 0$, since ${\bf v}$ is arbitrary. 
%
%\be
%\begin{split}
%0&\geq  \psi'(t_0)+ \alpha \sup_{{\bf v}\in \Delta(V)}\inf_{ P\in {\mathcal A}_x(t_0)}
%\int_{\R^{N}} \lg \p,
%\left[\int_{\U(t_0)\times V }\frac{1}{h}\int_{t_0}^{t_0+h} f(y, u(s),v)ds\ dP^{v}_y( u )d{\bf v}(v)\right]\rg d\bnu(y) \\
%&\qquad  -\|f\|_\infty \eta -Ch  \\
%&\geq  \psi'(t_0)+ \alpha \sup_{{\bf v}\in {\bf V}}
%\int_{\R^{N}} \inf_{ P\in {\mathcal A}_x(t_0)}\essinf}_{t\in [t_0,t_0+h]} \int_{\U(t_0)\times V }\lg \p,
% f(y,u(s),v) \rg dP^v(u)d{\bf v}(v)d\bnu(y)\\
%&\qquad  - \|f\|_\infty \eta -Ch  \\
%&\geq  \psi'(t_0)+ \alpha \sup_{{\bf v} \in \Delta(V) }  \int_{\R^{N}} \inf_{{\bf u}\in \Delta(U_n)} \lg -\p(y), 
%\int_{U\times  V}  f(y,u,v)\  d{\bf u}(u) d{\bf v}(v)\rg d \bnu(y)- \|f\|_\infty \eta -Ch
%\end{split}
%\ee
%Letting $h\to0$ gives  the desired inequality: $\V^+$ is a solution of the Hamilton-Jacobi equation \eqref{HJ}. The final condition is straightforward.
\end{proof}

%%%%%%%%%%%%%%%%%%%%%%%%%%%%%%%%%%%%%%%%%%%%%%%%%%%%%
\section{The discretized game}
%%%%%%%%%%%%%%%%%%%%%%%%%%%%%%%%%%%%%%%%%%%%%%%%%%%%%%
\label{section jeu discret}
In order to prove that the game has a value, we have to introduce some auxiliary discretized game for which  the existence of the value can be obtained
by classical min-max arguments.

\subsection{Discrete strategies of Player II}
\label{definition de Vn}
In this new game, the actions of Player II  are random controls defined on a suitable finite set. More precisely,  let us fix an integer
$n\geq 1$. Let  $\tau_n=\frac{T}{n}$ be the time step
and $t_i^n=i\tau_n$ (for $i=0, \dots, n$) be a grid on $[0,T]$. We consider an increasing family $(V_n)$ of finite subsets of $V$ such that, for any
$n\geq 1$ and any $v\in V$, there is some $v_n\in V_n$ with $|v-v_n|\leq 1/n$. For each $n\geq 1$ and $t_0\in [0,T]$, we denote by $\Vt_n(t_0)$ the finite subset 
of $\Vt(t_0)$ consisting in step functions with constant value on each interval $[t_i^n,t_{i+1}^n)$ and taking values in $V_n$. Let  $\Delta(\Vt_n(t_0))$ be  the set
of all probability measures over $\Vt_n(t_0)$. If we denote by $N_n$ the cardinal of $\Vt_n(t_0)$, then the set $\Delta(\Vt_n(t_0))$ can be identified with the simplex of $\R^{N_n}$ because 
each element of $\Delta(\Vt_n(t_0))$ can be written as $\tilde v=\sum_{i=1}^{N_n}p_iv_i$ for some $p\in(\R^+)^{N_n}$ such that $p\cdot\vec 1=1$. Then we set $d^\epsilon(\tilde v,\tilde v')=d^\epsilon(\sum_{i=1}^{N_n}p_iv_i,\sum_{i=1}^{N_n}p'_iv_i)=\|p-p'\|_1$. The set $\Delta(\Vt_n(t_0))$ can therefore
be viewed as a compact, convex subset of $\R^{N_n}$.

\subsection{Discretized game}
The discretized game is the game where Player I plays some strategy  $P\in \Delta( \A_x^{\tau_n}(t_0))$ and Player II plays
some random control  $\tilde v\in\Delta(\Vt_n(t_0))$.
Our aim is to use Sion's Theorem in order to prove that the discretized game has a value. For this we see $\Delta( \A_x(t_0))$ as a convex subset
of the vector space of the set of maps from $\R^N\times \Vt(t_0)$ into the set of Borel signed measures on $\U(t_0)$. We can 
endowed $\Delta( \A_x(t_0))$ with the distance 
$$
d(P, \tilde P) = \sup_{(x,v)\in \R^N\times \Vt(t_0)} d_{\Delta( \U(t_0))} (P^v_x,\tilde P^v_x)
$$
(recall that the distance $d_{\Delta( \U(t_0))}$ on $\Delta( \U(t_0))$ is defined in section \ref{Definitions and assumptions}). For any fixed $(t_0,\mu_0)\in[0,T]\times\W$, we note that the map $(P,Q)\to \I(t_0,\mu_0, P,Q)$ is linear with respect to $P$ and to $Q$
and continuous with respect to both variables on $\Delta( \A_x^{\tau_n}(t_0))$ and $\Delta( \Vt_n(t_0))$. 
Indeed the continuity with respect to $Q$ is obvious since $\Delta( \Vt_n(t_0))$ is finite dimensional. The continuity with respect to 
$P$  also holds because, since the map $u\to g(X_T^{t_0,x,u,v})$ is continuous on $\U(t_0)$ for any $(x,v)$, the map
$p\to \int_{\U(t_0)} g(X_T^{t_0,x,u,v})dp(u)$ is continuous on $\Delta( \U(t_0))$ for any $(x,v)$. The continuity of the map
$P \to \int_{\R^N\times \U(t_0)\times \Vt(t_0)}g(X_T^{t_0,x,u,v})dP^v_x(u)dQ(v)d\mu_0(x)$ on $\Delta( \A_x(t_0))$ 
then follows from Lebesgue dominate convergence Theorem since $g$ is bounded.

We can now use Sion's minmax Theorem to get:

\begin{Lemma}\label{discrete dyn prog}
For all $n\in\N^*$,  the discretized game on $\Delta( \A^{\tau_n}_x(t_0))\times \Delta( \Vt_n(t_0))$ has a value, denoted by $\V_n(t_0,\mu_0)$:
\begin{equation*}
\begin{split}
\V_n(t_0,\mu_0)&=\inf_{ P\in\Delta( \A_x^{\tau_n}(t_0))}\sup_{Q \in\Delta(\Vt_n(t_0))}\int_{\R^N\times \U(t_0)\times \Vt_n(t_0)}
g\left(X_{T}^{t_0,x, u,v}\right) dP^v_x(u)dQ(v)d\mu_0(x)\\
&=\inf_{ P\in\Delta( \A_x^{\tau_n}(t_0))}\sup_{v \in\Vt_n(t_0)}\int_{\R^N\times \U(t_0)}
g\left(X_{T}^{t_0,x, u,v}\right) dP^v_x(u)d\mu_0(x)\\
&=\sup_{Q \in\Delta(\Vt_n(t_0))}\inf_{ P\in\Delta( \A_x^{\tau_n}(t_0))} \int_{\R^N\times \U(t_0)\times \Vt_n(t_0)}
g\left(X_{T}^{t_0,x, u,v}\right) dP^v_x(u)dQ(v)d\mu_0(x)\\
&=\sup_{Q \in\Delta(\Vt_n(t_0))}\inf_{ \alpha\in \A_x^{\tau_n}(t_0)} \int_{\R^N\times \Vt_n(t_0)}
g\left(X_{T}^{t_0,x, \alpha(x,v) ,v}\right) dQ(v)d\mu_0(x) 
\end{split}
\end{equation*}
Moreover $\V_n$ is Lipschitz continuous in both variables uniformly with respect to $n$. 
\end{Lemma}
\begin{proof}
The Lipschitz continuity of $\V_n$ can be established as in Proposition \ref{V+V- Lipschitz}.
\end{proof}

Recalling the definition of $\V_\tau$ in (\ref{defV-tau}) one easily gets: \[\V_n(t_0,\mu_0)\leq \V^-_{\tau_n}(t_0,\mu_0)\;.\]
It remains to check that $\liminf_{n\rightarrow \infty}\V_n(t_0,\mu_0)\geq \V^+(t_0,\mu_0)$. This is the aim of the next section. 
For this we need two preliminary lemmas:
\begin{Lemma}
The value function $\V_n$ satisfies the dynamic programming principle:
\[\V_n(t^n_k,\mu)=\inf_{ P\in \Delta( \A_x^{\tau_n}(t^n_k))}\sup_{v \in\Vt_n(t^n_k)}\V_n(t^n_{k+1},\mu_{t^n_{k+1}}^{t^n_k,\mu, P,v})\;.\]
\end{Lemma}
\begin{proof} The proof is closely related to that of Proposition \ref{dyn prog}, so we only explain the main differences. 
Let us denote by $W(t^n_k,\mu)$ the right-hand side of the above equality. One can check, as in the proof of Proposition \ref{V+V- Lipschitz}, that
$W$ is Lipschitz continuous with respect to $\mu$. Inequality  $\V_n\geq W$ can be established as in Proposition \ref{dyn prog}.

Let us now prove the inequality $\V_n\leq W$. 
Let $\ep>0$ and $P_0$ be $\ep-$optimal for $W(t^n_k,\mu)$. 
Let us fix $\delta>0$ small and $v\in \Vt_n(t^n_k)$. We set $s_0=t^N_k$ and $s_1=t^N_{k+1}$. 
At time $s_1+\delta$, Player~I knows which constant control $v_i\in V_n$
 Player~II has been playing on the time interval $[s_0, s_1]$, so he or she knows the measure $\nu_i:=\mu_{s_1}^{s_0, \mu, P_0, v_i}$. 
Let $P_i$ be $\ep-$optimal for  $\V_n(s_1, \nu_i)$. We restrict the strategy $P_i$ to the time interval $[s_1+\delta,T]$ by setting
$$
\int_{\U(s_1+\delta)}\varphi(u_2)d\tilde P_{i,x}^{v}(u_2)
=
\int_{\U(s_1)}\varphi(u_{|_{[s_1+\delta,T]}}))d P_{i,x}^{v_{|_{[s_1,T]}}}(u)
$$
for any $(x,v)\in \R^N\times \Vt(s_0)$ and any nonnegative Borel measurable map $\varphi:\U(s_1+\delta)\to \R$. 
We finally define the strategy $P\in \Delta( \A_x(s_0))$ by using the identification $\U(s_0)= \U(s_0, s_1+\delta)\times \U(s_1+\delta)$:
$$
dP^v_x((u_1,u_2))=dP_{0,x}^v(u_1)\left( \sum_{v_i\in V_n} {\bf 1}_{\{v_{|_{[s_0,s_1]}=v_i}\}} 
d\tilde P^{v}_{i, X_{s_1}^{s_0,x,P_0, v_i }}(u_2)\right)\;.
$$
This means that Player~I plays
the strategy $P_0$ on the time interval $[s_0,s_1+\delta]$, and then switches at time $s_1+\delta$ to the strategy $\tilde P_i$ 
evaluated at the point $X_{s_1}^{s_0,x,P_0, v_i }$ if the control played by Player~II on $[s_0,s_1]$ has been $v_i$. 
It is then a routine computation to show that the strategy $P$ satisfies
$$
\I(s_0,\mu_\ep, P,v) \leq W(s_0,\mu)+C\ep+O(1)\;,
$$
where $O(1)\to 0$ as $\delta\to0$, uniformly with respect to $v$ and  we conclude as for Proposition  \ref{dyn prog} that $\V_n\leq W$.
\end{proof}

\begin{Lemma}\label{vn quasi sous solution}
The value $\V_n$ of the discretized game satisfies for all test function 
$$
\varphi(t,\mu)=-\frac{\alpha}{2} \d^2(\bmu,\mu)-\eta \d(\bnu,\mu)+\psi(t)
$$ 
(where $\psi:\R\to\R$ is smooth, $\alpha,\eta>0$ and $\bmu,\bnu\in \W$)
such that $\V_n-\varphi$ has a global minimum at $(t^n_k,\bnu)$ and for any optimal plan $\pi\in\Pi_{opt}(\bmu,\bnu)$:
\be
0\geq\psi'(t^n_k)+H_n(\bnu,\alpha\p)-\|f\|_\infty\eta-O\left(\frac{1}{n}\right)\;,
\ee
where $\p$ is defined by
\[\int_{\R^{2N}}\langle\xi(y),x-y\rangle d\pi(x,y)=\int_{\R^N}\langle\xi(y),\p(y)\rangle d\bnu(y) \qquad \forall \xi\in L^2_\bnu(\R^N,\R^N)\]
and where the Hamiltonian $H_n$ is given by:
\be
\label{Hn}
H_n(\mu,\p)=\max_{{\bf v}\in \Delta(V^{n})} \int_{\R^N}\min_{{\bf u}\in \Delta(U)} \int_{U\times V} \langle\p,f(y,u,v)\rangle d{\bf u}(u)d{\bf v}(v)\ d\mu(y).
\ee
\end{Lemma}
\begin{proof} We assume, without loss of generality, that $\V_n(t^n_k,\bnu)=\varphi(t^n_k,\bnu)$.
%Consider any test function as in the statement of the Lemma such that $\V_n-\varphi$ has a local minimum at $(t^n_k,\bnu)$.
%Now applying the dynamic programming principle satisfied by $\V_n$ for $\tau_n=\frac{T}{n}$ (previous Lemma) and assuming that $\V_n(t^n_k,\bnu)=\varphi(t^n_k,\bnu)$, we have:
Applying the dynamic programming principle of Lemma \ref{discrete dyn prog} we have
\ben
%\begin{split}
0\geq\inf_{ P\in\A_x^{\tau_n}(t^n_k)}\sup_{Q\in\Delta( \Vt_n(t^n_k))}[\varphi(t^n_k+\tau_n,\nu(t^n_k+\tau_n))-\varphi(t^n_k,\bnu)]
%&\geq \psi'(t^n_k)\tau_n+\alpha \tau_n \max_{V^{n}} \int_{\R^N}\inf_U\langle\p,f(y,u,v)\rangle d\bnu(y)-\|f\|_\infty\eta\tau_n-O\left(\tau_n^2\right)
%\end{split}
\een
where we have set \[\nu(t^n_k+\tau_n)=\mu_{t^n_k+\tau_n}^{t^n_k,\bnu, P,Q}\;.\]
Using the special form of $\varphi$, we get:
\ben
0\geq\inf_{ P\in\A_x^{\tau_n}(t^n_k)}\sup_{Q\in\Delta( \Vt_n(t^n_k))}\left[\psi(t^n_k+\tau_n)-\psi(t^n_k)+\frac{\alpha}{2} [\d^2(\bmu,\bnu)-\d^2(\bmu,\nu(t^n_k+\tau_n))]-\eta\d(\bnu,\nu(t^n_k+\tau_n)\right]\;.
\een
Arguing as in Section \ref{Caracterisation de la valeur solution HJ} we have:
\begin{multline*}
0\geq\tau_n \psi'(t^n_k) -O(\tau_n^2)\\+ \inf_{ P\in\A_x^{\tau_n}(t^n_k)}\sup_{Q\in\Delta( \Vt_n(t^n_k))}\left[\alpha\int_{\R^{N}\times \U(t^N_k)\times \Vt(t^N_k)} 
\lg \p(y),
\int_{t^N_k}^{t^N_{k+1}} f(y, u(s),v(s))ds \rg dP^v_y(u)dQ(v)d\bar \nu(y)-\eta\|f\|_\infty \tau_n\right].
\end{multline*}
We now note that the restriction of the strategy $Q$ to $[t^n_k,t^n_k+\tau]$ just consists in playing a constant control
at random according to some probability measure ${\bf v}\in \Delta(V_n)$. Moreover the strategy $P$, having for delay $\tau_n$, 
does not depend on $v$ on this time interval and therefore amounts to playing at random a control $u$ with probability $P_x(u)$
independent of $v$. Denoting by $\Delta( \U(t^n_k)))$ the set of probability measures on $\U(t^n_k)$ and using the min-max Theorem, we have
\ben
\begin{split}
0&\geq\tau_n \psi'(t^n_k) -\eta\|f\|_\infty \tau_n-O(\tau_n^2)\\
&\qquad+ \ds{ \alpha\max_{{\bf v}\in \Delta(V_n)}\int_{\R^{N}}\inf_{P\in\Delta( \U(t^n_k))} \int_{\U(t^n_k)\times V}\lg \p(y),
\int_{t^N_k}^{t^N_{k+1}} f(y, u(s),v) ds\rg dP(u)d{\bf v}(v)\ d\bnu(y)  }\\
%&\geq \tau_n \psi'(t^n_k) -\eta\|f\|_\infty \tau_n-O(\tau_n^2)+\alpha\max_{v\in V_n}\int_{\R^{N}}\inf_{u\in\U(t^n_k)} \lg \p(y),
%\int_{t_0}^{t_0+\tau_n} f(y,u(s),v)ds\rg d\bnu(y)\\
&\geq \ds{  \tau_n \psi'(t^n_k) -\eta\|f\|_\infty \tau_n-O(\tau_n^2) }\\
&\qquad+ \ds{ \alpha\tau_n\max_{{\bf v}\in \Delta(V_n)}
\int_{\R^{N}}\min_{{\bf u}\in \Delta(U)} \int_{U\times V} \lg \p(y),
 f(y,u,v)\rg d{\bf u}(u)d{\bf v}(v)\ d\bnu(y) \;. }\\
\end{split}
\een
%Thus dividing by $\tau_n$, we have:
%\be
%0\geq\psi'(t^n_k)+H_n(\bnu,\alpha\p)-\|f\|_\infty\eta-O\left(\frac{1}{n}\right)\;.
%\ee
\end{proof}

%%%%%%%%%%%%%%%%%%%%%%%%%%%%%%%%%%%%%%%%%%%%%%%%%%%%%%%%%%%%%%%%%%%%%%%%%%%%%%%%%%%%%%%%%%%%%%%
%%%%%%%%%%%%%%%%%%%%%%%%%%%%%%%%%%%%%%%%%%%%%%%%%%%%%%%%%%%%%%%%%%%%%%%%%%%%%%%%%%%%%%%%%%%%%%

\section{Existence and Characterization of the value}
%In this section, we prove that $\liminf \V_n$ is greater than or equal to the upper value. This implies that the upper and lower values are equal and characterized as the viscosity solution of the Hamilton-Jacobi equation \eqref{HJ} with final condition \eqref{condition finale}.\\
\begin{Theorem}
The game has a value, namely:
\[\V^+=\V^-\]
characterized as the unique viscosity solution of the Hamilton-Jacobi equation \eqref{HJ}.
\end{Theorem}
\begin{proof} We have already noticed that $\V_n(t_0,\mu_0)\leq \V^-_{\tau_n}(t_0,\mu_0)$.
It remains to check that $\liminf_{n\rightarrow \infty}\V_n(t_0,\mu_0)\geq \V^+(t_0,\mu_0)$. The main idea is to use the stability of
viscosity solutions as in \cite{Crandall Lions infty}.
From Proposition \ref{V+ visc sol} we know that $\V^+$ is a solution of the Hamilton-Jacobi equation (\ref{HJ}), while 
Lemma \ref{vn quasi sous solution} states that $\V_n$ is ``almost a subsolution" of that equation. Moreover, 
the functions $\V_n$ and $\V^+$ are bounded and Lipschitz continuous with the same Lipschitz constant denoted by $K$ and we have 
$\V^+(T,\mu)=\V_n(T,\mu)$ for all $\mu\in\W$.

Let us introduce the functions $\u_n(t,\mu)=e^{t}\V_n(t,\mu)$ and $\u^+(t,\mu)=e^{t}\V^+(t,\mu)$.
The function $\u^+$ is a viscosity solution of
\be
\label{HJ e-tv}
-\u+\u_t+H(\mu,D_\mu\u)=0
\ee
%in the sense of Section \ref{Caracterisation de la valeur solution HJ} 
with final condition $\u^+(T,\mu)=e^{T}\int_{\R^N}g(x)d\mu(x)$. For $(s,\mu), (t, \nu) \in [0,T]\times \W$, with 
$s\in \{t^n_i, \ i=0, \dots, n\}$, we set
 \[\varphi_{\ep}^n(s,\mu,t,\nu)=\u^+(t,\nu)-\u_n(s,\mu)-\frac{1}{\epsilon}\d^2(\mu,\nu)-\frac{1}{\epsilon}(t-s)^2\;.\]
The function $\varphi_{\ep}^n$ is continuous and bounded from above.
From Lemma \ref{Ekeland},
for all $\delta>0$, there exists $(\bs,\bmu,\bt,\bnu)$ 
\be
\label{Ekeland sur Phi n}
\varphi_{\ep}^n(\bs,\bmu,\bt,\bnu)\geq \varphi_{\ep}^n(s,\mu,t,\nu)-\delta[ \d(\mu,\bmu)+\d(\nu,\bnu)]\qquad \forall (s,\mu,t,\nu)
\ee
and
$$
\varphi_{\ep}^n(\bs,\bmu,\bt,\bnu)\geq \sup_{(s,\mu,t,\nu)} \varphi_{\ep}^n (s,\mu,t,\nu) -\delta\;.
$$
This auxiliary function gives a bound on $\sup_{(t,\mu)}[\V^+(t,\mu)-\V_n(t,\mu)]$:
\be
\label{majoration norme infinie v vn}
\sup_{(t,\mu)}[\V^+(t,\mu)-\V_n(t,\mu)]\leq \sup_{(t,\mu)}[\u^+(t,\mu)-\u_n(t,\mu)]\leq \varphi_{\ep}^n(\bs,\bmu,\bt,\bnu)+\delta\;.
\ee
Now we use the fact that
\[\varphi_{\ep}^n(\bs,\bmu,\bt,\bnu)\geq \varphi_{\ep}^n(\bs,\bmu,\bs,\bmu)-\delta \d(\bmu,\bnu)\;, \]
namely
\ben
\begin{split}
&\u^+(\bs,\bmu)-\u_n(\bs,\bmu)-\delta \d(\bmu,\bnu)\\
&\leq \u^+(\bt,\bnu)-\u_n(\bs,\bmu)-\frac{1}{\epsilon}\d^2(\bmu,\bnu)-\frac{1}{\epsilon}(\bt-\bs)^2\\
&\leq \u^+(\bs,\bmu)+ K |\bs-\bt|+K \d(\bnu,\bmu)-\u_n(\bs,\bmu)-\frac{1}{\epsilon}\d^2(\bmu,\bnu)-\frac{1}{\epsilon}(\bt-\bs)^2\;,
\end{split}
\een
to get the following bound on the distance between $(\bt,\bnu)$ and $(\bs,\bmu)$:% \[\epsilon (K+\delta)\d(\bnu,\bmu)+\epsilon K |\bt-\bs|\geq \d^2(\bnu,\bmu)+(\bt-\bs)^2\geq \frac{\left(\d(\bnu,\bmu)+|\bt-\bs|\right)^2}{2}\]
%and finally
\be
\label{estimation dmunu}
\d(\bnu,\bmu)+|\bt-\bs|\leq 2\epsilon(K+\delta)\;.
\ee
%for $\delta,\epsilon$ small enough.\\
Let us first assume that $\bs=T$ (the case $\bt=T$ could be treated similarly).
Then we have:
\ben
\begin{split}
\varphi_{\ep}^n(\bs,\bmu,\bt,\bnu)&\leq \u^+(T,\bnu)-\u_n(T,\bmu)+K|T-\bt|\\
&\leq e^T\int_{\R^N}g(x)d\bnu(x)-e^T\int_{\R^N}g(x)d\bmu(x)+K|T-\bt|\\
&\leq {\rm Lip}(g)e^T\d(\bnu,\bmu)+K|T-\bt|\; \leq \;  C \epsilon
\end{split}
\een
Thus using \eqref{majoration norme infinie v vn} we get:
\[\sup_{(t,\mu)}[\V^+(t,\mu)-\V_n(t,\mu)]\leq C\epsilon +\delta\;,\]
so that, passing to the limit as $\epsilon, \delta \rightarrow 0$, we obtain:
\[\sup_{(t,\mu)}[\V^+(t,\mu)-\V_n(t,\mu)]\leq 0\;.\]
We now assume that $\bt\neq T$ and $\bs\neq T$. Let us fix some optimal transport plan  $\pi\in\Pi_{opt}(\bnu,\bmu)$. 
We first use the fact that
\[\varphi_{\ep}(\bs,\bmu,\bt,\bnu)\geq \varphi_{\ep}(s,\mu,\bt,\bnu)-\delta\d(\mu,\bmu)\;, \]
namely
\ben
\begin{split}
&\u^+(\bt,\bnu)-\u_n(\bs,\bmu)-\frac{1}{\epsilon}\d^2(\bmu,\bnu)-\frac{1}{\epsilon}(\bt-\bs)^2\\
&\qquad\geq \u^+(\bt,\bnu)-\u_n(s,\mu)-\frac{1}{\epsilon}\d^2(\mu,\bnu)-\frac{1}{\epsilon}(\bt-s)^2 -\delta\d(\mu,\bmu)\;,
\end{split}
\een
to get
\ben
\u_n(s,\mu)+\frac{1}{\epsilon}\d^2(\mu,\bnu)+\delta\d(\mu,\bmu)+\frac{1}{\epsilon}(\bt-s)^2 \\
\geq \u_n(\bs,\bmu)+\frac{1}{\epsilon}\d^2(\bmu,\bnu)+\frac{1}{\epsilon}(\bt-\bs)^2\;.
\een
If we set $\varphi(s,\mu)=-\frac{1}{\epsilon}\d^2(\mu,\bnu)-\delta\d(\mu,\bmu)-\frac{1}{\epsilon}(\bt-s)^2$, then
the function $\u_n-\varphi$ has a minimum at $(\bs,\bmu)$. From Lemma \eqref{vn quasi sous solution},
this implies that:
\[ -\u_n(\bs,\bmu)+\frac{2}{\epsilon}(\bt-\bs)+H_n(\bmu,\frac{2}{\epsilon}\p)\leq\delta\|f\|_\infty+O\left(\frac{1}{n}\right)\]
where $\p$ is defined by:
\be\label{def py} 
\int_{\R^{2N}}\langle\xi(y),x-y\rangle d\pi(x,y)=\int_{\R^N}\langle\xi(y),\p(y)\rangle d\bmu(y) \qquad  \forall \xi\in L^2_\bmu(\R^N,\R^N)\;.
\ee
The same argument applied to
\[\varphi_{\ep}^n(\bs,\bmu,\bt,\bnu)\geq \varphi_{\ep}^n(\bs,\bmu,t,\nu)-\delta\d(\nu,\bnu)\]
leads to
%\ben
%\begin{split}
%&\u^+(\bt,\bnu)-\u_n(\bs,\bmu)-\frac{1}{\epsilon}\d^2(\bmu,\bnu)-\frac{1}{\epsilon}(\bt-\bs)^2\\
%&\geq \u^+(t,\nu)-\u_n(\bs,\bmu)-\frac{1}{\epsilon}\d^2(\bmu,\nu)-\frac{1}{\epsilon}(t-\bs)^2-\delta\d(\nu,\bnu)
%\end{split}
%\een
%This means that denoting by $\varphi(t,\nu)=\frac{1}{\epsilon}\d^2(\bmu,\nu)+\delta\d(\nu,\bnu)+\frac{1}{\epsilon}(t-\bs)^2$,
%the function $\u^+-\varphi$ has a maximum at $(\bt,\bnu)$.
%The function $\u^+$ being a subsolution, we get by definition:
\[ \u^+(\bt,\bnu)-\frac{2}{\epsilon}(\bt-\bs)-H(\bnu,-\frac{2}{\epsilon}\q)\leq \delta\|f\|_\infty\]
where $\q$ satisfies for $\bar\pi\in\Pi_{opt}(\bmu,\bnu)$:
\[\int_{\R^{2N}}\langle\xi(y),x-y\rangle d\bar\pi(x,y)=\int_{\R^N}\langle\xi(x),\q(x)\rangle d\bnu(x) \qquad \forall \xi\in L^2_\bnu(\R^N,\R^N)\;.\]
If we take for $\bar\pi$ the optimal transport plan defined through:
\[\int_{\R^{2N}}\varphi(x,y) d\bar\pi(x,y)=\int_{\R^N}\varphi(y,x)d\pi(x,y) \qquad \forall \varphi\in L^2_\pi(\R^{2N},\R^{2N})\]
we have:
\[\int_{\R^{2N}}\langle\xi(y),x-y\rangle d\pi(x,y)=\int_{\R^N}\langle\xi(x),-\q(x)\rangle d\bnu(x) \qquad \forall \xi\in L^2_\bnu(\R^N,\R^N)\;.\]
Finally combining the last two inequalities we obtain:
\be
\label{inegalite sur u+-un}
\u^+(\bt,\bnu)-\u_n(\bs,\bmu) \leq H(\bnu,-\frac{2}{\epsilon}\q)-H_n(\bmu,\frac{2}{\epsilon}\p)+2\delta\|f\|_\infty+O\left(\frac{1}{n}\right)\;.
\ee
Our next step consists in comparing $H$ and $H_n$:
\begin{Lemma}\label{HHn} We have, for any $\mu\in \W$ and any $\p\in L^2_\mu(\R^N,\R^N)$:
\[0\leq H(\mu,\p)-H_n(\mu,\p)\leq \gamma_n\|\p\|_{L^2_\mu} \;,\]
where 
$$
\gamma_n= \sup_{x,u, |v_1-v_2|\leq 1/n} |f(x,u,v_1)-f(x,u,v_2)|\;.
$$
\end{Lemma}
\begin{Remark}{\rm Note that $\gamma_n\to 0$ as $n\to+\infty$ because $f$ is uniformly continuous on $\R^N\times U\times V$. 
}\end{Remark}

\begin{proof} By definition we have $H(\mu,\p)\geq H_n(\mu,\p)$. Let $\bar {\bf v}$ be $\ep-$optimal for $H(\mu,\p)$ and $\Pi:V\to V_n$ be a 
Borel measurable selection of the projection map from $V$ onto $V_n$. Then, by construction of $V_n$, we have 
$|v-\Pi(v)|\leq 1/n$ and 
$$
\begin{array}{rl}
 H_n(\mu,\p)\;  \leq & \ds{  \int_{\R^N}\inf_{{\bf u}\in \Delta(U)} \int_{U\times V} \lg f(x,u,v),\p(x)\rg d{\bf u}(u)d\Pi\sharp \bar {\bf v}(v)\ d\mu(x) } \\
 \leq & \ds{ \int_{\R^N}\inf_{{\bf u}\in \Delta(U)} \int_{U\times V} \lg f(x,u,v),\p(x)\rg d{\bf u}(u)d\bar {\bf v}(v)\ d\mu(x) +
 \gamma_n \int_{\R^N} |\p(x)|  d\mu(x) } \\
\leq & H(\mu, \p) + \ep+ \gamma_n\|\p\|_{L^2_\mu}
\end{array}
$$
%
%We have by definition
%\ben
%\begin{split}
%0\leq H(\mu,\p)-H_n(\mu,\p)&=\max_{V} \int_{\R^N}\inf_U\langle\p,f(y,u,v)\rangle d\mu(y)-\max_{V^{n}} \int_{\R^N}\inf_U\langle\p,f(y,u,v)\rangle d\mu(y)\\
%&=\max_{V} \int_{\R^N}\langle\p,f_v(y)v\rangle d\mu(y)-\max_{V_n} \int_{\R^N}\langle\p,f_v(y)v\rangle d\mu(y)\\
%%&\leq \int_{\R^N}\langle\p,f_v(y)v_\epsilon\rangle d\mu(y)-\max_{V^{n}} \int_{\R^N}\langle\p,f_v(y)v\rangle d\mu(y)+\epsilon\\
%\end{split}
%\een
%We consider some $\epsilon$-optimal point in $V$ for the first term of the right-hand side of the last equation and denote it by $v_\epsilon$ to get:
%\[H(\mu,\p)-H_n(\mu,\p)\leq \int_{\R^N}\langle\p,f_v(y)v_\epsilon\rangle d\mu(y)-\max_{V^{n}} \int_{\R^N}\langle\p,f_v(y)v\rangle d\mu(y)+\epsilon\;.\]
%Now we can choose $v_n\in V_n$ such that $\|v_\epsilon-v_n\|\leq \frac{1}{n}$ by definition of $V_n$ c.f. Subsection \ref{definition de Vn}.
%Therefore:
%\ben
%\begin{split}
%H(\mu,\p)-H_n(\mu,\p)&\leq \int_{\R^N}\langle\p,f_v(y)(v_\epsilon-v_n)\rangle d\mu(y)+\epsilon\\
%&\leq \int_{\R^N}\langle x-y,f_v(y)(v_\epsilon-v_n)\rangle d\pi(x,y)+\epsilon\\
%&\leq \left( \int_{\R^{2N}} \|f_v(y)(v_\epsilon-v_n)\|^2d\pi\right)^{\frac{1}{2}}\times\left( \int_{\R^{2N}} \|x-y\|^2d\pi\right)^{\frac{1}{2}}+\epsilon\\
%&\leq \frac{\|f\|_\infty}{n}\d(\mu,\nu)+\epsilon\\
%\end{split}
%\een
%This inequality being true for all $\epsilon$, we have the desired inequality.
\end{proof}

From (\ref{def py}), we have 
$$
\|\p\|_{L^2_{\mu}} \leq \d(\bmu,\bnu)\;.
$$
Combining the continuity of the Hamiltonian stated in Lemma \ref{continuity of the hamiltonian}, the bound on the distance between $H$ and $H_n$ 
given in Lemma \ref{HHn} and inequality \eqref{inegalite sur u+-un} we get:
\[\u^+(\bt,\bnu)-\u_n(\bs,\bmu) \leq  \frac{2K}{\epsilon}\d^2(\bnu,\bmu)+ \frac{2\gamma_n}{\ep}\d(\bnu,\bmu)+2\delta\|f\|_\infty+O\left(\frac{1}{n}\right)\;.\]
Putting together the estimates  \eqref{majoration norme infinie v vn} and  \eqref{estimation dmunu} and letting $\epsilon,\delta\to 0$ finally gives
\[\V^+\leq \V_n+ C\gamma_n+ O\left(\frac{1}{n}\right)\;.\]
This implies that $\liminf_{n\rightarrow \infty}\V_n(t_0,\mu_0)\geq \V^+(t_0,\mu_0)$ and completes the proof. 
%
%Now recall that we had in Section \ref{section jeu discret} for all $n$ and all $(t,\mu)$ denoting by $\tau_n=\frac{T}{n}$:
%\[\V_n(t,\mu)\leq \V_{\tau_n}^-(t,\mu)\]
%namely
%\[\liminf_{n\rightarrow \infty}\V_n(t,\mu)\leq \V^-(t,\mu)\;.\]
%Combined to the fact that
%\[\V^+\leq \V_n+ O\left(\frac{1}{n}\right)\]
%and
%\[\V^+\geq\V^-\]
%we can conclude that
%\[\V^+=\V^-\;.\]
%Therefore, the game admits a value and this value is the viscosity solution of the Hamilton-Jacobi equation \eqref{HJ} satisfying the final condition \eqref{condition finale}.
\end{proof}
%\newpage

\section*{Appendix}
\label{ekeland}
The following statement is a slight modification of Ekeland's variational Lemma.
\begin{Lemma}\label{Ekeland}
Let $F:[0,T]\times \W\to \R$ be a  continuous function which is bounded from below. Then for any $\epsilon>0$, there exists $(\bt,\bmu)\in X$
such that for all $(t,\mu)\in X$:
\[F(t,\mu)\geq F(\bt,\bmu)-\epsilon \d(\mu,\bmu)\qquad {\rm and }\qquad F(\bt,\bmu) \leq \inf_X F+\ep\; .\]
\end{Lemma}
\begin{proof}
Let $(t_0,\mu_0)$ be such that \[F(t_0,\mu_0)\leq \inf_XF+\epsilon\;.\]
Then we build the sequence $(t_n,\mu_n)$ by induction, such that, if $(t_n,\mu_n)$ is known, then
\begin{itemize}
\item if for all $(t,\mu)\in X$, $F(t,\mu)\geq F(t_n,\mu_n)-\epsilon \d(\mu,\mu_n)$, then we set $(t_{n+1},\mu_{n+1})=(t_n,\mu_n)$, 
\item if, on the contrary,  there is $(t,\mu)\in X$ such that $F(t,\mu)< F(t_n,\mu_n)-\epsilon \d(\mu,\mu_n)$, then we set 
$$
S_n=\{(t,\mu)\in X \; \mbox{\rm  such that $F(t,\mu)< F(t_n,\mu_n)-\epsilon \d(\mu,\mu_n)$ }\}
$$ and choose
$(t_{n+1},\mu_{n+1})\in S_n$ such that $F(t_{n+1},\mu_{n+1})\leq (F(t_n,\mu_n)+ \inf_{S_n}F)/2$.
\end{itemize}
Note that by construction $F(t_n,\mu_n)\leq \inf_XF+\ep$ for any $n$. 
We prove that the sequence $(\mu_n)$ is a Cauchy sequence.
Indeed, it is either stationary, or we have $\epsilon \d(\mu_n,\mu_{n+1})< F(t_n,\mu_n)-F(t_{n+1},\mu_{n+1})$.
Therefore for all $(n,p)$, $n\geq p$,
\be
\label{cauchy}
\epsilon \d(\mu_n,\mu_{p})< F(t_{p},\mu_{p})-F(t_n,\mu_n)\;.
\ee
The sequence $F(t_n,\mu_n)$ being decreasing and bounded from below, it has a limit and inequality (\ref{cauchy}) shows that 
$(\mu_n)$ is a Cauchy sequence.
Let $\bmu$ be the limit of the $(\mu_n)$ and let us consider any cluster point $\bt$ of the $(t_n)$.
We now assume for a while that the is some $(\bs,\bnu)\in X$ with
\be\label{def bs bnu}
F(\bs,\bnu)< F(\bt,\bmu)-\epsilon \d(\bnu,\bmu)\;.
\ee
Consider some subsequence $(t_{n_i},\mu_{n_i})$ converging to $(\bt,\bmu)$. Letting $n\to +\infty$ in  \eqref{cauchy} gives
\[F(\bt,\bmu)\leq F(t_{n_i},\mu_{n_i})-\epsilon \d(\mu_{n_i},\bmu)\;.\]
Therefore, we have
\[F(\bs,\bnu)< F(t_{n_i},\mu_{n_i})-\epsilon \d(\bnu,\mu_{n_i})\;,\]
which means that $(\bs,\bnu)\in S_{n_i}$ for all $i$. %Note that the sequence $S_n$ is decreasing. This means that $(\bs,\bnu)\in S_{n_i}\subset S_{n_i-1}$
This implies that \[2F(t_{n_{i}+1},\mu_{n_{i}+1})-F(t_{n_i},\mu_{n_i})\leq \inf_{S_{n_i}}F\leq F(\bs,\bnu)\;.\]
The sequence $F(t_n,\mu_n)$ being decreasing, we get: $2F(t_{n_{i+1}},\mu_{n_{i+1}})-F(t_{n_i},\mu_{n_i})\leq  F(\bs,\bnu)$.
Passing to the limit as $i\rightarrow \infty$ gives $F(\bt,\bmu)\leq  F(\bs,\bnu)$, which is in contradiction with (\ref{def bs bnu}).
Therefore we have 
\[F(t,\mu)\geq F(\bt,\bmu)-\epsilon \d(\mu,\bmu) \qquad \forall (t,\mu)\in X\;. \]

\end{proof}

%%%%%%%%%%%%%%%%%%%%%%%%%%%%%%%%%%%%%%%%%%%%%%%%%%%%
% BIBLIO
%%%%%%%%%%%%%%%%%%%%%%%%%%%%%%%%%%%%%%%%%%%%%%%%%%%%

\end{document}